\pdfoutput=1 
\documentclass{amsart}
\usepackage[utf8]{inputenc}
\usepackage[T1]{fontenc}
\usepackage{graphicx}
\usepackage{longtable}
\usepackage{wrapfig}
\usepackage{rotating}
\usepackage[normalem]{ulem}
\usepackage{amsmath}
\usepackage{amssymb}
\usepackage{capt-of}
\usepackage{hyperref}
\usepackage{etoolbox}
\usepackage{microtype}
\usepackage{mathrsfs}
\usepackage[english]{babel}
\usepackage{tikz}
\usepackage{tikz-cd}
\usepackage{amsthm}
\usepackage{thmtools}
\usepackage{thm-restate}
\usepackage{xpatch}
\usepackage[autostyle]{csquotes}
\declaretheorem[numberwithin=section, style=plain]{theorem}  
\declaretheorem[sibling=theorem, style=plain]{corollary, lemma, proposition, conjecture}
\declaretheorem[sibling=theorem, style=definition]{definition, example}
\declaretheorem[sibling=theorem, style=remark]{remark, notation}
\declaretheorem[numbered=no, style=remark]{acknowledgements}

\usepackage[backend=biber,giveninits=true,url=true,eprint=true,doi=false,isbn=false]{biblatex}
\renewbibmacro{in:}{%
  \ifentrytype{article}{}{\printtext{\bibstring{in}\intitlepunct}}}

\tikzcdset{
  cells={font=\everymath\expandafter{\the\everymath\displaystyle}},
}

\DeclareMathOperator\id{id}

\let\hom\undefined
\DeclareMathOperator\hom{Hom}

\DeclareMathOperator\ext{Ext}

\newcommand\dualab\widehat
\newcommand\sh\mathscr

\makeatletter
\newcounter{proofstep}
\xpretocmd{\proof}{\setcounter{proofstep}{0}}{}{}
\newcommand{\proofstep}[1]{%
  \par
  \addvspace{\medskipamount}%
  \stepcounter{proofstep}%
  \noindent\emph{Step \theproofstep: #1}\par\nobreak\smallskip
  \@afterheading
}
\makeatother

\makeatletter
\let\save@mathaccent\mathaccent
\newcommand*\if@single[3]{%
  \setbox0\hbox{${\mathaccent"0362{#1}}^H$}%
        \setbox2\hbox{${\mathaccent"0362{\kern0pt#1}}^H$}%
  \ifdim\ht0=\ht2 #3\else #2\fi
  }
\newcommand*\rel@kern[1]{\kern#1\dimexpr\macc@kerna}
\newcommand*\widebar[1]{\@ifnextchar^{{\wide@bar{#1}{0}}}{\wide@bar{#1}{1}}}
\newcommand*\wide@bar[2]{\if@single{#1}{\wide@bar@{#1}{#2}{1}}{\wide@bar@{#1}{#2}{2}}}
\newcommand*\wide@bar@[3]{%
  \begingroup
  \def\mathaccent##1##2{%
    \let\mathaccent\save@mathaccent
    \if#32 \let\macc@nucleus\first@char \fi
    \setbox\z@\hbox{$\macc@style{\macc@nucleus}_{}$}%
    \setbox\tw@\hbox{$\macc@style{\macc@nucleus}{}_{}$}%
    \dimen@\wd\tw@
    \advance\dimen@-\wd\z@
    \divide\dimen@ 3
    \@tempdima\wd\tw@
    \advance\@tempdima-\scriptspace
    \divide\@tempdima 10
    \advance\dimen@-\@tempdima
    \ifdim\dimen@>\z@ \dimen@0pt\fi
    \rel@kern{0.6}\kern-\dimen@
    \if#31
      \overline{\rel@kern{-0.6}\kern\dimen@\macc@nucleus\rel@kern{0.4}\kern\dimen@}%
      \advance\dimen@0.4\dimexpr\macc@kerna
      \let\final@kern#2%
      \ifdim\dimen@<\z@ \let\final@kern1\fi
      \if\final@kern1 \kern-\dimen@\fi
    \else
      \overline{\rel@kern{-0.6}\kern\dimen@#1}%
    \fi
  }%
  \macc@depth\@ne
  \let\math@bgroup\@empty \let\math@egroup\macc@set@skewchar
  \mathsurround\z@ \frozen@everymath{\mathgroup\macc@group\relax}%
  \macc@set@skewchar\relax
  \let\mathaccentV\macc@nested@a
  \if#31
    \macc@nested@a\relax111{#1}%
  \else
    \def\gobble@till@marker##1\endmarker{}%
    \futurelet\first@char\gobble@till@marker#1\endmarker
    \ifcat\noexpand\first@char A\else
      \def\first@char{}%
    \fi
    \macc@nested@a\relax111{\first@char}%
  \fi
  \endgroup
}
\makeatother

\newcommand\Dt{\widetilde{\mathscr{D}}}
\newcommand\D{\mathscr{D}}
\newcommand\Mt{\widetilde{\M}}
\newcommand\M{\mathscr{M}}
\newcommand\Dp[1]{{}_{\mathrm{D}}#1_{\ast}}
\newcommand\Dpc[2]{\Dp{#1}^{(#2)}}
\newcommand\Dpb[1]{{}_{\mathrm{D},\widebar{\mathrm{D}}}#1_{\ast}}
\newcommand\Qt{\widetilde{\Theta}}
\DeclareMathOperator{\Sp}{Sp}
\newcommand\Sptrxy{\Sp_{X\to Y}}
\newcommand\Sptrxyb{\Sp_{X,\widebar{X}\to Y,\widebar{Y}}}
\newcommand\cur{\mathfrak{C}}
\NewDocumentCommand\Tp{m}{{}_{\mathrm{T}}\csname push:nn\endcsname{#1}{_}}
\usepackage{xparse,mathtools}
\ExplSyntaxOn
\cs_new_protected:Npn \push:nn #1 #2
{
\tl_set:Nn \l_tmpa_tl { #1 }
\tl_if_in:NnTF \l_tmpa_tl { #2 }
{ \tl_replace_all:Nnn \l_tmpa_tl { #2 } { \pushsb:n } }
{ \tl_put_right:Nn \l_tmpa_tl { \sb { \ast } } }
\tl_use:N \l_tmpa_tl
}
\cs_new_protected:Npn \pushsb:n #1
{ \sb { #1 \ast } }
\ExplSyntaxOff
\newcommand\Tpc[2]{\Tp{#1}^{(#2)}}
\newcommand\T{\widetilde{\mathscr{T}}}
\newcommand\Sptrxyinf{\Sptrxy^{\infty}}
\newcommand\Sptrxybinf{\Sptrxyb^{\infty}}
\newcommand\Et{\widetilde{\mathcal{E}}}
\newcommand\Ot{\widetilde{\mathcal{O}}}
\newcommand{\kt}{\widetilde{K}}
\newcommand{\hmq}[2]{\operatorname{HM}_{\mathbb{Q}}(#1,#2)^p}
\newcommand{\hmz}[2]{\operatorname{HM}_{(\mathbb{Z})}(#1,#2)^p}
\address{Department of Mathematics, Humboldt-Universität zu Berlin}
\email{mads.villadsen@hu-berlin.de}
\author{Mads Bach Villadsen}
\date{}
\title{Hodge modules and Kähler morphisms}
\hypersetup{
 pdfauthor={Mads Bach Villadsen},
 pdftitle={Hodge modules and Kähler morphisms},
 pdfkeywords={},
 pdfsubject={},
 pdflang={English},
 final,
 bookmarks=true,
 bookmarksnumbered=true,
 bookmarksdepth=subsubsection}

\begin{document}

\maketitle
\begin{abstract}
We prove the decomposition theorem for Hodge modules with integral structure along proper Kähler morphisms, partially generalizing M. Saito's theorem for projective morphisms. Our proof relies on compactifications of period maps of generically defined variations of Hodge structure, as well as the theorems of Cattani-Kaplan-Schmid on \(L^2\)-cohomology of variations of Hodge structure for the hard Lefschetz theorem.
\end{abstract}
\section{Introduction}
\label{sec:org6877fcc}
One of the main results of M. Saito's theory of Hodge modules is a Hodge-theoretic proof of the decomposition theorem of Beilinson-Bernstein-Deligne-Gabber \autocite{Beilinson1982}, at least for those semisimple perverse sheaves underlying polarizable pure Hodge modules. However, while parts of the theory works just as well on arbitrary complex manifolds as on projective varieties, Saito's proof of the decomposition theorem \autocite{Saito1988} relies heavily on the assumption that the morphism in question be projective. In particular, the argument uses dimensional induction to prove a relative hard Lefschetz theorem by using a version of the Lefschetz hyperplane theorem for holonomic \(\mathscr D\)-modules.

Since the usual hard Lefschetz theorem for singular cohomology still holds in the compact Kähler setting, it's natural to ask whether the decomposition theorem holds along proper relative Kähler morphisms, suitably defined. We use geometric methods, inspired by Saito's original argument, to prove this for polarizable complex Hodge modules which admit an integral structure, in particular for Hodge modules of geometric origin (see \autoref{sec:org3b1eb0c}).
\begin{theorem}
\label{thm:decomp}
Let \(f\colon X\to Y\) be a proper morphism of complex manifolds, and suppose \(l\in H^2(X,\mathbb{R}(1))\) is a relative Kähler class. Then for any \(M\in \hmz X n\) a polarizable pure \(\mathbb{C}\)-Hodge module of weight \(n\) which admits an integral structure, we have the following:
\begin{enumerate}
\item \(\Tp f M\) is strict.
\item \(\Tpc f k M\in \hmz Y {n+k}\)
\item The relative hard Lefschetz theorem holds, i.e. \(l^k\colon \Tpc f {-k}M\to \Tpc f kM(k)\) is an isomorphism for every \(k\ge 0\).
\item If \(M\) is polarized by the sesquilinear pairing \(\mathcal{S}\colon M\otimes \widebar{M}\to \cur_X\), then \(\Tpc f {-k}M\) is polarized by
\begin{align*}
(\alpha,\widebar{\beta})\mapsto \Tpc f {k,-k} \mathcal{S}(l^k\alpha,\widebar{\beta}).
\end{align*}
\end{enumerate}
\end{theorem}
Here \(\Tpc f k\) denotes the \(k\)th derived direct image functor of triples of \(\Dt\)-modules (see \autoref{sec:org6cd7cae}) and \(a_X\colon X\to \{pt\}\) the constant map. A class \(l\in H^2(X,\mathbb{R}(1))\) is said to be relatively Kähler along \(f\) if, locally over \(Y\), there exists a closed 2-form \(\widetilde{l}\) representing \(l\) and a Kähler form \(\eta\) on the base such that \(\widetilde{l}+f^{\ast}\eta\) is a Kähler form (see also \autoref{sec:org3a41ff2}).
An integral structure on a variation of Hodge structures is just a local system with integer coefficients that recovers the complex local system of the VHS; in general, an integral structure is a constructible complex with integer coefficients.

Note that if \(K_{\mathbb{R}}\) is a real perverse sheaf which provides a real structure for \(M\), then the hard Lefschetz isomorphism on \(M\) is induced by a corresponding isomorphism \(l^k\colon {}^pR^{-k}f_{\ast}K_{\mathbb{R}}\to {}^pR^kf_{\ast}K_{\mathbb{R}}\).

If \(f\colon X\to Y\) is a projective morphism, our methods in fact apply to arbitrary polarizable Hodge modules, not just those which admit an integral structure (see \autoref{rmk:projective}). This strictly generalizes \autocite[Théorème 5.3.1]{Saito1988}, since the Kähler cone may be bigger than the ample cone on projective manifolds.

\begin{corollary}
\label{cor:rational-decomp}
With notation and assumptions as in \autoref{thm:decomp}, suppose \(F\) is a subfield of \(\mathbb{C}\) and that \(K_F\) is a perverse sheaf over \(F\) underlying \(M\), i.e. that there is given an isomorphism between \(K_F\otimes_F \mathbb{C}\) and the de Rham complex of the \(\D\)-module underlying \(M\). Then there is a decomposition \(\mathbf{R}f_{\ast}K_F\simeq\bigoplus_k{}^p\sh H^k\mathbf{R}f_{\ast}K_F[-k]\) in the derived category.
\end{corollary}
The decomposition over \(F\) is obtained by taking a decomposition over either \(\mathbb{R}\) or \(\mathbb{C}\) depending on whether \(F\) is a real field, and then deforming this decomposition using density of \(F\) in the larger field.
The perverse sheaf \(K_F\) is not required to be compatible with the Hodge filtration on \(M\) in any way, hence the decomposition of \(K_F\) will not be compatible with the Hodge module structure. Even if \(M\) is an \(F\)-Hodge module, however, the decomposition of \(K_F\) will still not generally be a decomposition of \(F\)-Hodge modules, since the deformation procedure is not compatible with the Hodge filtrations. It is possible that there is some other decomposition of \(K_F\) which arises from a decomposition of \(F\)-Hodge modules, but there seems to be no a priori reason why such a decomposition should exist unless the map \(f\) is projective.

Kollár in \autocite{Kollar1986a} stated a series of conjectures about the existence of certain coherent sheaves associated to generically defined VHS, satisfying certain properties enjoyed by the canonical bundle and its higher direct images. In the projective case these conjectures were proven by Saito \autocite{Saito1991}. As a consequence of \autoref{thm:decomp}, we can prove parts of those conjectures for Kähler morphisms and integral VHS, see \autoref{thm:kollar}.
\subsection{Outline of the proof}
\label{sec:org758c477}
The central idea in this paper is to use period maps to replace the Lefschetz pencils in Saito's proof in the projective case. If one tries to imitate Saito's argument in the Kähler setting, at some point one wants to factor the constant map \(a_X\colon X\to \{\mathrm{pt}\}\) through some intermediate map \(f\colon X\to Y\) to allow for induction on the fibre dimension of the morphism. In the projective setting there are plenty of rational maps available for this purpose, but a general compact Kähler manifold will not admit any, even meromorphic, maps to lower dimensional varieties. However, in Saito's argument, \(X\) comes equipped with a VHS defined on a Zariski open subset \(U\) of \(X\). Since we assume an integral structure on our Hodge modules, this VHS will also carry an integral structure, giving rise to a period map defined on \(U\). The case where the period map is trivial can be handled directly, and if the period map is non-trivial (more precisely, if it has infinite monodromy), we show that, up to replacing \(X\) with a generically finite cover, \(X\) admits a meromorphic map to \(\mathbb{P}^1\). Here we use a theorem essentially due to Sommese \autocite{Sommese1978} to compactify the period map (see \autoref{thm:sommese-compactification}), and a theorem of Brunebarbe and Cadorel \autocite{Brunebarbe2020} to further map from the image of the compactified period map to \(\mathbb{P}^1\) (\autoref{cor:VHS-implies-algebraicity}).

The methods of this paper would apply to Hodge modules with more general coefficients if one could generalize this construction of meromorphic functions to the more general setting.

We will need some input from analysis in two parts of the proof. The first part is for the case of the constant map from a curve, where we use Zucker's results on \(L^2\)-cohomology on a curve \autocite{Zucker1979}; note that Saito's proof in the projective case proceeds in the same manner. The other part is in proving the hard Lefschetz theorem for the cohomology of a Hodge module supported on a compact subvariety of a Kähler manifold. For this, Mochizuki proves a reduction lemma \autocite[Section 6.3]{Mochizuki2022}, implying in our setting that hard Lefschetz for arbitrary Hodge modules can be reduced to hard Lefschetz for intersection cohomology of VHS defined on the complement of a normal crossings divisor, which follows from the \(L^2\)-analysis of \autocite{Cattani1987,Kashiwara1987}. Mochizuki's reduction lemma is in essence a (very complicated) lemma in linear algebra, so the analytic input to our proof comes entirely from the classical results of \autocites{Kashiwara1987}[][]{Cattani1987}[][]{Zucker1979}.

Due to the difficulty of handling certain signs in Saito's implementation of Hodge modules, we instead follow the implementation of complex Hodge modules given by the MHM project of Sabbah-Schnell \autocite{Sabbah2018}, which relies only on D-modules and does not use perverse or constructible sheaves in the definition. This is particularly important in the proof of \autoref{lemma:P1-factorization}. See \autoref{sec:org1bcedfc} for details.
\subsection{Related work}
\label{sec:orgf07e9fd}
The results of this paper are special cases of Mochizuki's paper \autocite{Mochizuki2022} which proves a more general version of the decomposition theorem for Kähler morphisms for tame twistor \(\mathscr D\)-modules \autocite{Mochizuki2022}, hence in particular covering the case of complex Hodge modules. The proof given here is significantly simpler, however, and should be easier to follow for anyone familiar with Saito's original argument in the projective case. Further, this paper applies to all Hodge modules of geometric origin, hence to a large fraction of the Hodge modules encountered in practice.

It is a theorem of Cattani-Kaplan-Schmid and Kashiwara-Kawai \autocite{Cattani1987,Kashiwara1987} that the intersection cohomology of a variation of Hodge structures defined on the complement of a normal crossings divisor can be computed using \(L^2\)-cohomology with respect to an appropriately chosen metric. Since one can prove the Kähler identities for \(L^2\)-forms, this implies that the intersection cohomology satisfies hard Lefschetz, and endows it with a polarizable Hodge structure. Mochizuki's main new result is that this \(L^2\)-Hodge filtration coincides with the proposed Hodge filtration defined in terms of filtered \(\mathscr D\)-modules by Saito (rather, Mochizuki proves a twistor version of this result, but we will stick to Hodge modules in this paper. Note also that the VHS version of this theorem was the never-published result announced in \autocite{Kashiwara1986}).

We note that the reduction lemma of \autocite[Section 6.3]{Mochizuki2022} that we use in the present paper does not depend on any of the new analytic results that Mochizuki proves.

Separate from the results on \(L^2\)-cohomology,
Saito \autocite{Saito1990a} claimed to prove parts of \autoref{thm:decomp} (though for arbitrary real Hodge modules) conditional on results announced by Kashiwara and Kawai \autocite{Kashiwara1986}; however, these announced results were never published (see also \autocite[Remarks, end of Section 2.5]{Saito2016}). Saito later \autocite{Saito2022} gave a corrected proof for the decomposition theorem for the constant sheaf.

\begin{acknowledgements}
I thank Christian Schnell for bringing my attention to this problem and many useful discussions; Claude Sabbah for answering my questions regarding the MHM project and polarizations, and detailed feedback on a draft of this paper; Johan Commelin for his interest in the issue of signs in polarizations of Hodge modules in the sense of Saito; and Thomas Krämer for reading and giving feedback on a draft.
\end{acknowledgements}
\section{Preliminaries}
\label{sec:org3fe6ab7}
\subsection{Hodge structures, filtrations and polarizations}
\label{sec:org9a97fba}
In this section, we spell out some terminology and conventions regarding complex Hodge structures and polarizations, following \autocite{Sabbah2018}. We will use Deligne's sign convention \autocite{Deligne1971} (used also by both the MHM project and Saito \autocite{Saito1988}); note that this differs from the papers of Cattani-Kaplan-Schmid \autocite{Cattani1987,Cattani1986} by a factor \((-1)^n\) for pure weight \(n\) Hodge structures \autocite[Remark 2.5.15]{Sabbah2018}. This difference will be relevant in proving \autoref{lemma:mhs-polarization-reduction-of-variables}.

Throughout this section, let \(\mathcal{H}\) be a finite-dimensional complex vector space. A pure \(\mathbb{C}\)-Hodge structure of weight \(w\) on \(\mathcal{H}\) is the data \(H=(\mathcal{H},F'^{\bullet}\mathcal{H},F''^{\bullet}\mathcal{H})\), where \(F'\) and \(F''\) are \(w\)-opposite decreasing filtrations, i.e. \(\mathcal{H}\cong \bigoplus_p\mathcal{H}^{p,w-p}\) where \(\mathcal{H}^{p,w-p}=F'^p\mathcal{H}\cap F''^{w-p}\mathcal{H}\).
If \(R\subseteq \mathbb{R}\) is a subring, then an \(R\)-Hodge structure on a free \(R\)-module \(\mathcal{H}\) is a \(\mathbb{C}\)-Hodge structure on \(\mathcal{H}_{\mathbb{C}}\) where additionally \(F''^{w-p}= \widebar{F'^p}\).

A mixed Hodge structure is the data \(H=(\mathcal{H},F'^{\bullet}\mathcal{H},F''^{\bullet}\mathcal{H},W_{\bullet}\mathcal{H})\), where \(W_{\bullet}\) is now an increasing exhaustive filtration such that the filtrations induced by \(F'\) and \(F''\) on \(\operatorname{gr}_j^W\mathcal{H}\) define a pure Hodge structure of weight \(j\) for each \(j\in \mathbb{Z}\). For a ring \(R\subseteq \mathbb{R}\), a mixed \(R\)-Hodge module is a free \(R\)-module \(\mathcal{H}\) consists of a mixed Hodge structure on \(\mathcal{H}_{\mathbb{C}}\) such that the weight filtration \(W_{\bullet}\) is defined over \(R\), and such that the Hodge structures on the associated graded pieces of \(W_{\bullet}\) are \(R\)-Hodge structure.

Let \(\mathbb{Z}(w)\) be the unique weight \(-2w\) integral Hodge structure on \((2\pi i)^w\mathbb{Z}\) (by which we mean an \(\mathbb{R}\)-Hodge structure on \((2\pi i)^w\mathbb{Z}\otimes_{\mathbb{Z}}\mathbb{R}\)); let \(\mathbb{R}(w)\) and \(\mathbb{C}(w)\) be the associated real and complex Hodge structures. For any trifiltered vector space \(H=(\mathcal{H},F'^{\bullet},F''^{\bullet},W_{\bullet})\), let \(H(w)=H\otimes_{\mathbb{Z}}\mathbb{Z}(w)\) with the induced filtrations. Here we consider any pure Hodge structure \(K\) of weight \(w\), and in particular \(\mathbb{Z}(w)\), as a mixed Hodge structure with \(\operatorname{gr}_w^WK=K\).

Given a Hodge structure, let \(C_D\colon \mathcal{H}\to \mathcal{H}\) be the linear automorphism given by multiplication by \((-1)^q\) on \(\mathcal{H}^{p,q}\); this is the Deligne-Weil operator associated to the Hodge structure \(H\).

\begin{definition}
Let \(H\) be a pure Hodge structure of weight \(w\) as above. A \emph{polarization} of \(H\) is a morphism of Hodge structures \(S\colon H\otimes \widebar{H}\to \mathbb{C}(-w)\) such that
\begin{enumerate}
\item \(S\) is hermitian, i.e. \(\widebar{S(x,\widebar{y})}=S(y,\widebar{x})\) for all \(x,y\in \mathcal{H}\), and
\item the pairing \(h(x,\widebar{y})=S(C_Dx,\widebar{y})=S(x, \widebar{C_Dy})\) on \(\mathcal{H}\) is hermitian positive definite.
\end{enumerate}
\end{definition}

While one can define rational and real polarizations \autocite[Section 2.5.c]{Sabbah2018}, we will not consider these in this paper.

Given an \(\mathfrak{sl}_2\)-representation on \(\mathcal{H}\), by the associated \(\mathfrak{sl}_2\)-triple we will mean the images \((\mathrm{X,H,Y})\) of the standard generators of \(\mathfrak{sl}_2\), where \(\mathrm{X}\) increases weights, \(\mathrm{Y}\) lowers weights, and \(\mathrm{H}\) is semisimple.

\begin{definition}
An \(\mathfrak{sl}_2\)-Hodge structure of central weight \(w\) on \(\mathcal{H}\) consists of filtrations \(F',F''\) and an \(\mathfrak{sl}_2\)-triple \((\mathrm{X,H,Y})\) such that
\begin{enumerate}
\item Each eigenspace \(\mathcal{H}_k\) of \(\mathrm{H}\), \(k\in \mathbb{Z}\), underlies a pure Hodge structure \(H_k\) of weight \(w+k\) with filtrations induced by \(F',F''\),
\item \(\mathrm{X}\colon H_k\to H_{k+2}(1)\) and \(\mathrm{Y}\colon H_k\to H_{k-2}(-1)\) are morphisms of Hodge structures.

A polarization of this \(\mathfrak{sl}_2\)-Hodge structure is a morphism of mixed Hodge structures
\begin{align}
S\colon H\otimes \widebar{H}\to \mathbb{C}(-w),
\end{align}
where the weight filtration on \(W\) is given by \(W(Y)\) (equivalently, it is induced by the grading by the \(H_k\)), such that
\begin{equation}\label{sl2:adj}
S(\mathrm{H}x,\widebar{y})=-S(x, \widebar{\mathrm{H}y}),\qquad S(\mathrm{X}x, \widebar{y})=S(x, \widebar{\mathrm{X}y}),\qquad S(\mathrm{Y}x, \widebar{y})=S(x, \widebar{\mathrm{Y}y})
\end{equation}
and the induced form
\begin{align}
S_{-k}\colon H_{-k}\otimes \widebar{H_{-k}}&\to \mathbb{C}(-w+k)\\
x\otimes \widebar{y} &\mapsto S(\mathrm{X}^kx,\widebar{y})\nonumber
\end{align}
polarizes the pure Hodge structure \(PH_{-k}\) for every \(k\ge 0\), where
\begin{align}
PH_{-k}=\ker(H_{-k}\xrightarrow{\mathrm{X}^{k+1}}H_{k+2})=\ker(H_{-k}\xrightarrow{\mathrm{Y}} H_{-k-2})
\end{align}
is the primitive subspace.
\end{enumerate}
\end{definition}

The polarization condition is equivalent to asking that \((-1)^kS(\mathrm{Y}^k x,\widebar{y})\) polarizes \(PH_k\) for every \(k\ge 0\) \autocite[3.2.10-11]{Sabbah2018}. Note also that specifying an \(\mathfrak{sl}_2\)-representation on \(\mathcal{H}\) is equivalent to giving a grading \(\mathcal{H}=\bigoplus_{k}\mathcal{H}_k\) and a graded operator \(\mathcal{H}\to \mathcal{H}[2]\) or \(\mathcal{H}\to \mathcal{H}[-2]\). The grading corresponds to the eigenspaces of the semisimple operator in an \(\mathfrak{sl}_2\)-triple, the graded operator to one of the nilpotent operators, and the other nilpotent operator is uniquely determined from this.

\begin{definition}
Let \(F',F''\) be filtrations on \(\mathcal{H}\) and \((\mathrm{X}_1,\mathrm{H}_1,\mathrm{Y}_1),(\mathrm{X}_2,\mathrm{H}_2,\mathrm{Y}_2)\) two commuting \(\mathfrak{sl}_2\)-triples. Write \(\mathcal{H}_{i,j}\) for the intersection of the \(i\)-eigenspace of \(\mathrm{H}_1\) and the \(j\)-eigenspace of \(\mathrm{H}_2\), so \(\mathcal{H}=\bigoplus_{i,j} \mathcal{H}_{i,j}\). We call this a bi-\(\mathfrak{sl}_2\)-Hodge structure of central weight \(w\) if for any \(k\), the data \((\bigoplus_j \mathcal{H}_{k,j},F',F'',(\mathrm{X}_2,\mathrm{H}_2,\mathrm{Y}_2))\) and \((\bigoplus_i\mathcal{H}_{i,k},F',F'',(\mathrm{X}_1,\mathrm{H}_1,\mathrm{Y}_1))\) define \(\mathfrak{sl}_2\)-Hodge structures of central weight \(w+k\).

A polarization of this bi-\(\mathfrak{sl}_2\)-Hodge structure is a morphism of mixed Hodge structures
\begin{align}
S\colon H\otimes \widebar{H}\to \mathbb{C}(-w)
\end{align}
(with weight filtration induced by the total grading of \(\mathcal{H}\)) such that the relations of Eq. \eqref{sl2:adj} are satisfied for both \(\mathfrak{sl}_2\)-triples, and the induced form
\begin{align}
S_{-i,-j}\colon \mathcal{H}_{-i,-j}\otimes \overline{\mathcal{H}_{-i,-j}}&\to \mathbb{C}(-w+i+j)\\
x\otimes \widebar{y}&\mapsto S(X_1^iX_2^jx,\widebar{y})\nonumber
\end{align}
polarizes the biprimitive Hodge structure \(P_1P_2H_{-i,-j}\) for all \(i,j\ge 0\).
\end{definition}

Given filtrations \(F',F''\) of \(\mathcal{H}\), we will, by slight abuse of notation, talk about nilpotent operators \(\mathrm{N}\colon H\to H(-1)\), by which we mean morphisms such that \(\mathrm{N}^l\colon H\to H(-l)\) vanishes for \(l\gg 0\). Equivalently, \(2\pi i \mathrm{N}\) is a nilpotent endomorphism of \(\mathcal{H}\) such that \(\mathrm{N}F^p\subset F^{p-1}\) for all \(p\), where \(F\) is either \(F'\) or \(F''\). For such an operator, we get an associated weight filtration \(W(\mathrm{N})\) on \(\mathcal{H}\). If \(H\) is a rational or real Hodge structure and \(\mathrm{N}\) is a morphism of rational or real Hodge structures, then \(2\pi i \mathrm{N}\) is a rational resp. real endomorphism, so \(W(\mathrm{N})\) is defined over \(\mathbb{Q}\) resp. \(\mathbb{R}\).

\begin{definition}
A Hodge-Lefschetz structure of central weight \(w\) on \(\mathcal{H}\) consists of filtrations \(F',F''\) and a nilpotent operator \(\mathrm{N}\colon H\to H(-1)\) such that \((\mathcal{H},F',F'',W(\mathrm{N}))\) is a mixed Hodge structure of central weight \(w\), or equivalently, \(\operatorname{gr}^{W(\mathrm{N})}\mathcal{H}\), with the filtrations and \(\mathfrak{sl}_2\)-triple induced by \(F',F''\) and \(\mathrm{N}\) (acting as the weight-decreasing operator \(\mathrm{Y}\)), is an \(\mathfrak{sl}_2\)-Hodge structure of central weight \(w\).

A polarization of \((\mathcal{H},F',F'',\mathrm{N})\) is a morphism of mixed Hodge structures
\begin{align}
S\colon H\otimes \widebar{H}\to \mathbb{C}(-n)
\end{align}
with respect to which \(\mathrm{N}\) is self-adjoint, and such that the induced pairing on \(\operatorname{gr}^{W(\mathrm{N})}\mathcal{H}\) polarizes the induced \(\mathfrak{sl}_2\)-Hodge structure.
\end{definition}

\begin{definition}
Let \((\mathcal{H},F',F'',W)\) be a mixed Hodge structure of central weight \(w\).

Given commuting nilpotent operators \(\mathrm{N}_1,\ldots,\mathrm{N}_k\colon H\to H(-1)\), let
\begin{align}
\mathrm{C}=\mathrm{C}(\mathrm{N}_1,\ldots,\mathrm{N}_k)= \{\sum_l \lambda_l\mathrm{N}_l\mid \lambda_l>0\}
\end{align}
be the open (real) cone generated by the \(\mathrm{N}_l\).

We say that \((\mathcal{H},F',F'',W)\) is polarized by a pairing \(S\) and the cone \(\mathrm{C}\) if \((\mathcal{H},F',F'',\mathrm{N},S)\) is a polarized Hodge-Lefschetz structure for every \(\mathrm{N}\in \mathrm{C}\).
\end{definition}

\begin{remark}
\label{rmk:lefschetz}
Usually in the theory of Hodge modules, the terminology introduced here is used in the case where \(\mathrm{N}\) is the logarithm of the unipotent part of a monodromy operator, where \(\mathrm{N}\) indeed decreases the weight by \(2\). In this paper, we will instead be concerned with Lefschetz operators \(\mathrm{L}\) given by cup product with a Kähler form, or more generally a closed \((1,1)\)-form, which increases weights by \(2\). However, in those cases \(H\) will be real and \(\mathbb{R}\)-split, i.e. will come with a splitting \(H=\bigoplus H_j\) of the weight filtration where each \(H_j\) is a real pure Hodge structure of weight \(j\), and \(\mathrm{L}\) is graded in the sense that \(2\pi i \mathrm{L}H_j\subset H_{j+2}(1)\). We can then replace \((\mathcal{H},F',F'')\) with \((\mathcal{H},\widetilde{F'},\widetilde{F''})=\bigoplus_j (\mathcal{H}_j,F',F'')(j)\), and then \(L\) acting on \((\mathcal{H},\widetilde{F'},\widetilde{F''})\) will decrease weights by \(2\), i.e. \((2\pi i)^{-1}\mathrm{L}\colon (\mathcal{H},\widetilde{F'},\widetilde{F''})\to (\mathcal{H},\widetilde{F'},\widetilde{F''})(-1)\). Thus we will allow ourselves to talk about Hodge-Lefschetz structures polarized by nilpotent real endomorphisms \(\mathrm{L}\) that increase the weight, as long as we have a splitting of the weight filtration. An equivalent trick is also used in \autocite[End of §3]{Cattani1987}.

Note that despite the grading on \(H\), we will still need to talk about Hodge-Lefschetz structures. For example, in \autoref{lemma:alteration} below, we need to consider two commuting nilpotent operators, with a given splitting of the weight filtration of each, where the corresponding \(\mathfrak{sl}_2\)-triples do not commute.
\end{remark}

The following is \autocite[Proposition 3.2.26]{Sabbah2018}.
\begin{lemma}
\label{lemma:bisl2}
Suppose given a bi-\(\mathfrak{sl}_2\)-Hodge structure of central weight \(w\) as above. Then the triple \((\mathrm{X}_1+\mathrm{X}_2,\mathrm{H}_1+\mathrm{H}_2,\mathrm{Y}_1+\mathrm{Y}_2)\) defines an \(\mathfrak{sl}_2\)-Hodge structure on \((\mathcal{H},F',F'')\). The corresponding grading of \(H\) is the total grading \(H=\bigoplus_k H^k\) defined by \(H^k=\bigoplus_{i+j=k}H_{i,j}\). If \(S\) polarizes the bi-\(\mathfrak{sl}_2\)-Hodge structure, it also polarizes the associated \(\mathfrak{sl}_2\)-Hodge structure.
\end{lemma}

It follows in particular that \((\mathcal{H},F',F'')\), seen as a mixed Hodge structure with weight filtration induced by the total grading, is polarized by \((S,\mathrm{C}(\mathrm{N}_1,\mathrm{N}_2))\). The following lemma explains how to go the other way; applied inductively, it allows to pass from a mixed Hodge structure polarized by a cone \(\mathrm{C}(\mathrm{N}_1,\ldots,\mathrm{N}_k)\) to an \(\mathfrak{sl}_2^k\)-Hodge structure (generalizing bi-\(\mathfrak{sl}_2\)-Hodge structures in the obvious way) by passing to the associated graded along the weight filtrations of the \(\mathrm{N}_i\). We will use this in the case \(k=2\) in \autoref{lemma:alteration} below.

\begin{lemma}
\label{lemma:mhs-polarization-reduction-of-variables}
Let \((\mathcal{H},F',F'',W)\) be a mixed \(\mathbb{R}\)-Hodge structure of central weight \(w\) polarized by \((S,\mathrm{C}(\mathrm{N}_1,\ldots,\mathrm{N}_k))\), where each \(\mathrm{N}_i\) is a real nilpotent operator.

Fix \(h\in \mathbb{Z}\) and define \((\widetilde{\mathcal{H}},\widetilde{F'},\widetilde{F''},\widetilde{W})\) as \(\widetilde{\mathcal{H}}=P_{\mathrm{N}_1}\operatorname{gr}_h^{W(\mathrm{N}_1)}\mathcal{H}\) with the induced filtrations. Let \(\widetilde{\mathrm{C}}=\mathrm{C}(\widetilde{\mathrm{N}_2},\ldots,\widetilde{\mathrm{N}_k})\) where \(\widetilde{\mathrm{N}_i} =\operatorname{gr}_h^{W(\mathrm{N}_1)}\mathrm{N}_i\), and let \(\widetilde{S}=S(\mathrm{N}_1^{-h},\id)\) (where \(\mathrm{N}_1^{-h}\) is the inverse of the hard Lefschetz isomorphism \(\mathrm{N}_1^h\) if \(h>0\)). Then \((\widetilde{\mathcal{H}},\widetilde{F'}, \widetilde{F''},\widetilde{W})\) is a mixed Hodge structure polarized by \((\widetilde{S},\widetilde{\mathrm{C}})\).
\end{lemma}
Recall from \autocite[(2.11)]{Cattani1987} that \((\mathcal{H},F',F'',W)\) is a mixed Hodge structure polarized by \((S,\mathrm{C}(\mathrm{N}_1,\ldots,\mathrm{N}_k))\) if and only if \(W=W(\mathrm{N})\) for every \(\mathrm{N}\in \mathrm{C}(\mathrm{N}_1,\ldots,\mathrm{N}_k)\) and the map
\begin{align}
\theta\colon  \mathbb{C}^k &\to \widehat{D}\\
(z_1,\ldots,z_k)&\mapsto \exp\left(\sum z_i\mathrm{N}_i\right)\cdot F'\nonumber
\end{align}
is a nilpotent orbit (see \autocite{Cattani1986}) where \(\widehat{D}\) is the flag manifold of filtrations on \(\mathcal{H}\) of type \(F'\); i.e. \(\mathrm{N}_iF'^p\subset F'^{p-1}\) and there exists some \(\alpha\in \mathbb{R}\) such that \(\theta\) is the period map associated to a pure polarized variation of Hodge structures on
\begin{align}
\{(z_1,\ldots,z_k)\in \mathbb{C}^k\mid \operatorname{Re}z_i<\alpha\}
\end{align}

Note that as we are now dealing with \(\mathbb{R}\)-Hodge structures, we only need to describe the single filtration \(F'\).

Note that \autocite{Cattani1987,Cattani1986} use nilpotent real endomorphisms, corresponding to \(2\pi i \mathrm{N}_j\) for \(\mathrm{N}_j\) as above, hence the condition \(\operatorname{Re}z_i<\alpha\) here rather than the condition \(\operatorname{Im}z_i>\alpha\).
\begin{proof}
Given \((\mathcal{H},F',F'',W;S,\mathrm{C})\) as in the lemma, let \(\theta\) be the associated nilpotent orbit as above.

Similarly, we can define a map
\begin{align}
\widetilde{\theta}\colon \mathbb{C}^{k-1}\to \widetilde{\widehat{D}}\colon (z_2,\ldots,z_k)\mapsto\exp\left(\sum z_i\widetilde{\mathrm{N}_i}\right)\cdot \widetilde{F'}
\end{align}
associated to \((\widetilde{\mathcal{H}},\widetilde{F'}, \widetilde{F''},\widetilde{W};\widetilde{S},\widetilde{\mathrm{C}})\), where \(\widetilde{\widehat{D}}\) is the flag manifold associated to \(\widetilde{\mathcal{H}}\) with filtration \(\widetilde{F'}\). It suffices to show that \(\widetilde{\theta}\) is a nilpotent orbit.

That \(\widetilde{\theta}\) is horizontal follows from \(\mathrm{N}_iF'^p\subset F'^{p-1}\). Fix now a \(p=(p_2,\ldots,p_k)\in \mathbb{C}^{k-1}\) with \(\operatorname{Re}p_i <\alpha\). Then the map \(z\mapsto \theta(z,p)\) is a nilpotent orbit in one variable. The associated mixed Hodge structure on \(\mathcal{H}\) has weight filtration \(W(\mathrm{N}_1)\), Hodge filtration \(\exp(\sum_{i=2}^kp_i\mathrm{N}_i)\cdot F'\), and is polarized by \((S,\mathrm{N}_1)\). The induced polarized pure Hodge structure on \(P_{\mathrm{N}_1}\operatorname{gr}_h^{W(\mathrm{N}_1)}\) is exactly the filtration \(\widetilde{\theta}(p)\) polarized by \(\widetilde{S}\), showing that \(\widetilde{\theta}\) is indeed a nilpotent orbit.
\end{proof}

\begin{remark}
\label{sl2-orbit-complex-vhs}
\autoref{lemma:mhs-polarization-reduction-of-variables} for \(\mathbb{C}\)-Hodge structures follows from the combination of the papers \autocites{Cattani1986}[][]{Sabbah2022} by the same argument. Namely, \autocite{Sabbah2018} proves the existence of limiting mixed Hodge structures for a VHS on a punctured disk in this setting, \autocite{Cattani1986} proves the \(\operatorname{SL}(2)\)-orbit theorem starting from the existence of the limiting mixed Hodge structures, and \autocite[(2.11)]{Cattani1987} for \(\mathbb{C}\)-Hodge structures follows (see \autocite[(4.66)]{Cattani1986}).
\end{remark}
\subsection{Kähler forms and Chern classes}
\label{sec:org3a41ff2}
As in \autocite{Deligne1971,Saito1988}, we will define Chern classes of line bundles on \(X\) using the exponential exact sequence
\begin{align}
0\to \mathbb{Z}(1)\to \mathcal{O}_X\xrightarrow{\exp} \mathcal{O}_X^{\times}\to 1,
\end{align}
letting \(c_1\colon H^1(X,\mathcal{O}_X^{\times})\to H^2(X,\mathbb{Z}(1))\) be the connecting homomorphism; hence Chern classes of line bundles will take values in \(\mathbb{Z}(1)=2\pi i \mathbb{Z}\).

Similarly, by a Kähler class we will mean a cohomology class \(\omega\in H^2(X,\mathbb{R}(1))\) such that the real cohomology class \((2\pi i)^{-1}\omega\) can be represented by a positive closed real \((1,1)\)-form on \(X\).
\subsection{Sesquilinear pairings and triples of \(\Dt\)-modules}
\label{sec:orge5d15b6}
The main objects in this paper are polarizable \(\mathbb{C}\)-Hodge modules with \(\mathbb{Z}\)-structure. We refer to \autocite{Sabbah2018} for the theory \(\mathbb{C}\)-Hodge modules and polarizations of such, but we will briefly recall some notation and terminology here.

We will work with right \(\D\)-modules. The sheaf \(\D_X\) is an algebra with filtration \(F\) defined by the order of each differential operator, hence gives rise to the associated Rees ring
\begin{align}
R_F\D_X=\bigoplus_pF_p\D_X\cdot z^p\subseteq\D_X\otimes_{\mathbb{C}}\mathbb{C}[z,z^{-1}],
\end{align}
a graded algebra. A filtered \(\D\)-module \(\M\) similarly gives rise to a graded \(R_F\D\)-module \(R_F\M\).

The next few definitions will make sense and be needed over both \(\D\) and \(R_F\D\). We will use \(\Dt\) as a symbol to denote either, distinguishing only when necessary. Similarly, \(\Mt\) will refer to a module over \(\Dt\). We follow the notation and constructions of \autocite{Sabbah2018} throughout.

Let \(\Mt\) be a right \(\Dt_X\)-module, \(\Qt_X\) the tangent sheaf of \(X\) as a right \(\Dt_X\)-module, and \(\Qt_{X,k}=\bigwedge^k\Qt_X\). The Spencer complex of \(\Mt\) is
\begin{align}
\Sp(\Mt)=\left( \Mt\otimes_{\widetilde{\mathcal{O}_X}}\Qt_{X,n}\xrightarrow{\widetilde\delta_{\Mt}}\Mt\otimes_{\widetilde{\mathcal{O}_X}}\Qt_{X,n-1}\to\cdots\to\Mt \right),
\end{align}
where the final term \(\Mt\) is in degree \(0\). The differential \(\widetilde\delta_{\Mt}\) is defined by
\begin{align*}
m\otimes(\xi_1\wedge\cdots\wedge\xi_k)\mapsto &\sum_{i=1}^k(-1)^{i-1}(m\xi_i)\otimes(\xi_1\wedge\cdots\wedge \widehat{\xi_i}\wedge\xi_k) \\
+&\sum_{i<j}(-1)^{i+j}m\otimes([\xi_i,\xi_j]\wedge\xi_1\wedge\cdots\wedge\widehat{\xi_i}\wedge\cdots\wedge\widehat{\xi_j}\wedge\cdots\wedge\xi_k).
\end{align*}

For a proper holomorphic map \(f\colon X\to Y\), define the relative Spencer complexes \(\Sptrxy(\Dt_X)=\Sp(\Dt_X)\otimes_{\widetilde{\mathcal{O}}_{X}}\Dt_{X\to Y}\) and \(\Sptrxy(\Mt)=\Mt\otimes_{\Dt_X}\Sptrxy(\Dt_X)\) for a right \(\Dt_X\)-module \(\Mt\).
The pushforward of \(\Mt\) is defined as
\begin{align}
\Dp\Mt = \mathbf{R}f_{\ast}(\Mt\otimes^{\mathbf{L}}_{\Dt_X}\Dt_{X\to Y})=\mathbf{R}f_{\ast}(\Mt\otimes_{\Dt_X}\Sptrxy(\Dt_X));
\end{align}
we denote the cohomology modules by \(\Dpc f j\Mt\).

Let \(\D_{X,\widebar{X}}=\D_X\otimes_{\mathbb{C}}\D_{\widebar{X}}\), and let \(\cur_X\) be the sheaf of currents of degree 0 on X, as a right \(\D_{X,\widebar{X}}\)-module.
\begin{definition}
A \emph{sesquilinear pairing} between right \(\D_X\)-modules \(\M',\M''\) on a complex manifold \(X\) is a \(\D_{X,\widebar{X}}\)-linear morphism
\begin{align}
\mathfrak{s}\colon \M\otimes_{\mathbb{C}} \widebar{\M''}\to \cur_X.
\end{align}

The Hermitian adjoint \(\mathfrak{s}^{\ast}\colon \M''\otimes_{\mathbb{C}} \widebar{\M'}\to \cur_X\) is defined by \(\mathfrak{s}^{\ast}(m'', \widebar{m'})= \widebar{\mathfrak{s}(m', \widebar{m''})}\).
\end{definition}

Let \(f\colon X\to Y\) be a proper holomorphic map. To define the pushforward of \(\mathfrak{s}\) along \(f\), let \(\Sptrxyb=\Sptrxy(\D_X)\otimes_{\mathbb{C}} \widebar{\Sptrxy(\D_X)}\), and note that we can compute the pushforward of \(\cur_X\) as
\begin{align}
\Dpb f\cur_X=\mathbf{R}f_{\ast}(\cur_X\otimes_{\D_{X,\widebar X}}\Sptrxyb),
\end{align}
which then comes with a natural morphism \(\Dpb f\cur_X\xrightarrow{\int_f}\cur_Y\) given by integration of currents.

Now we get morphisms
\begin{align}
&(\M'\otimes_{\D_X}\Sptrxy(\D_X))\otimes_{\mathbb{C}}( \widebar{\M''}\otimes_{\D_X}\Sptrxy(\D_X)) \\
\xrightarrow\sim &(\M' \otimes_{\mathbb{C}} \widebar{\M''})\otimes_{\D_{X,\widebar{X}}}\Sptrxyb\nonumber\\
\xrightarrow{\mathfrak{s}\otimes\id}&\cur_X\otimes_{\D_{X,\widebar{X}}}\Sptrxyb.\nonumber
\end{align}
Note that the first map, changing the order of the tensor product, does not introduce any sign since \(\M''\) is just a single object. However, when we later deal with pairings of complexes, this map will give rise to a sign that we must take into account.

Applying \(\mathbf{R}f_{\ast}\), composing with the integration map \(\int_f\) and taking cohomology objects gives, for each \(k\), the pushforward pairing
\begin{align}
\Dpb f^{(k,-k)}\mathfrak{s}\colon \Dpc f k \M'\otimes_{\mathbb{C}} \widebar{\Dpc f {-k}\M''}\to\Dpb f^{(0)}\cur_X\xrightarrow{\int_{f}}\cur_Y.
\end{align}

We will see later that it is more natural to add a sign to this definition.
\begin{definition}
The signed pushforward of \(\mathfrak{s}\) is \(\Tpc f {k,-k} \mathfrak{s}=\epsilon(k)\Dpb f^{(k,-k)}\mathfrak{s}\), where as usual \(\epsilon(k)=(-1)^{k(k-1)/2}\).
\end{definition}

From now on, we specifically work in the filtered setting and let \(\Dt_X=R_F\D_X\). Hodge modules will be certain \(\Dt\)-triples.
\begin{definition}
A triple on \(X\) is the data \(\T=(\Mt',\Mt'',\mathfrak{s})\), where \(\Mt',\Mt''\) are \(\Dt_X\)-modules and \(\mathfrak{s}\) is a sesuilinear pairing between \(\M',\M''\), where in general \(\M=\Mt/(z-1)\Mt\) is the associated \(\D_X\)-module.
\end{definition}
The constructions above give the cohomological pushforwards
\begin{align}
\Tpc f k \T=(\Dpc f k \Mt',\Dpc f {-k} \Mt'',\Tpc f {k,-k} \mathfrak{s}).
\end{align}

To construct a total pushforward complex \(\Tp f \T\) of which these are the cohomology objects, we follow \autocite[Section 12.7.d]{Sabbah2018} and work with the relative \(C^{\infty}\) Spencer complex \(\Sptrxyinf(\Mt)\) as a flabby resolution of \(\Sptrxy(\Mt)\), so we can compute the pushforward directly.
Recall from \autocite[8.4.13]{Sabbah2018} that the \(\mathcal{C}^{\infty}\) Spencer complex \(\widetilde{\Sp}_X^{\infty,i,j}\) is the double complex with terms \(\Qt_{X,i}\otimes_{\Ot_X}\Et^{(0,j)}\) in bidegree \((-i,j)\), where \(\Et^{(0,j)}\) is the sheaf of \(\mathcal{C}^{\infty}\) differential forms of type \((0,j)\). The differentials are given by
\begin{align*}
\Qt_{X,i}\otimes_{\Ot_X}\Et^{(0,j)}&\xrightarrow{\widetilde{\delta}'^{\infty}} \Qt_{X,i-1}\otimes_{\Ot_X}\Et^{(0,j)} \\
\xi\otimes\phi &\mapsto \widetilde{\delta}(\xi_i)\otimes\phi+\xi\lrcorner \widetilde{d}'\phi
\end{align*}
and \(\widetilde{d}''\), where \(\widetilde{d}',\widetilde{d}''\) are the usual holomorphic and antiholomorphic exterior derivatives on differential forms.

The relative \(\mathcal{C}^{\infty}\) Spencer resolution of \(\Mt\) can now be written as
\begin{align}
\Sptrxyinf(\Mt)=\Mt\otimes_{\Ot_X}\widetilde{\Sp}_X^{\infty,\bullet}\otimes_{\Ot_X}\Dt_{X\to Y},
\end{align}
where we take the simple complex associated to the double complex \(\widetilde{\Sp}_X^{\infty,\bullet,\bullet}\) defined above.

Thus we will let \(\widetilde K^{\bullet}=f_{\ast}\Sptrxyinf(\Mt)\), and define \(\Tp f \T\) as the complex of triples with terms
\begin{align}
\left(\kt'^k,\kt''^{-k},\epsilon(k)\int_ff_{\ast}\mathfrak{s}\right),
\end{align}
with the pairing \(\mathfrak{s}\) implicitly extended to \(C^{\infty}\)-coefficients.

The following is \autocite[Corollary 12.7.37]{Sabbah2018}
\begin{lemma}
Consider proper maps \(X\xrightarrow f Y\xrightarrow g Z\) and a triple \(\T\) on \(X\). Then the complexes \(\Tp {(g\circ f)} \T\) and \(\Tp g \Tp f \T\) (more precisely, the simple complex associated to the double complex that naturally computes the latter) are quasi-isomorphic, and in particular the functors \(\Tp f\) admit Leray spectral sequences in the category of triples.
\end{lemma}

We refer to \autocite{Sabbah2018} for the details of the proof, but since this commutation, which essentially boils down to a sign comparison, plays a crucial role in this paper (specifically in \autoref{lemma:P1-factorization}), let us explain the role of \(\epsilon(k)\) in this computation.

Write \(h=g\circ f\). To compute the pushforward \(\Tp h \T\), one considers, for \(\Mt=\Mt'\) or \(\Mt''\), the pushforwards
\begin{align}
\kt^{\bullet}=h_{\ast}\left(\Mt\otimes_{\Dt_X}\otimes\Sp_{X\to Z}^{\infty}(\Dt_X)\right),
\end{align}
together with the pairing between their underlying \(\D\)-modules defined from \(\mathfrak{s}\), taking values in
\begin{align}
h_{\ast}\left( \cur_X\otimes_{\Dt_{X,\widebar{X}}}\Sp^{\infty}_{X,\widebar{X}\to Y,\widebar{Y}}(\Dt_{X,\widebar{X}}) \right).
\end{align}
To get the final pushforward pairing, compose this with the integration morphism \(\int_h\).

To compute the double pushforward \(\Tp g \Tp f \T\), however, one must consider the double complexes
\begin{align}
\kt^{\bullet,\bullet}=g_{\ast}\left[ f_{\ast}\left( \Mt\otimes_{\Dt_X}\Sptrxyinf(\Dt_X) \right)\otimes_{\Dt_Y}\Sp^{\infty}_{Y\to Z} \right],
\end{align}
for \(\Mt=\Mt',\Mt''\), and the pairing between them induced by \(\mathfrak{s}\) taking values in
\begin{align}
g_{\ast}\left[ f_{\ast}\left( \cur_X\otimes_{\Dt_{X,\widebar{X}}}\Sptrxybinf(\Dt_{X,\widebar{X}}) \right)\otimes_{\Dt_{Y,\widebar{Y}}}\Sp^{\infty}_{Y,\widebar{Y}\to Z,\widebar{Z}} \right].
\end{align}

Recall now that the pushforward of \(\mathfrak{s}\), say along \(h\), is defined from \(\mathfrak{s}\colon\M'\otimes\M''\to\cur_X\) by combining the reordering of tensor products
\begin{align}
\M'\otimes\Sp(\D_X)\otimes \widebar{\M''\otimes\Sp(\D_X)}\xrightarrow\sim \M'\otimes \widebar{\M''}\otimes\Sp_{X,\widebar{X}\to Z,\widebar{Z}}(\D_{X,\widebar{X}})
\end{align}
with \(\mathfrak{s}\) and \(h_{\ast}\). For a single triple \(\T\), there is no sign in this tensor swap morphism. But in the computation of \(\Tp g\Tp f\T\), the pushforward \(\Tp g\) is applied to the complex of triples \(\Tp f\T\), and the tensor swap operation thus introduces a sign to the effect of \((-1)^{ij}\) on the pairing \(\mathfrak{s}_{ij}\colon K'^{i,j}\otimes \widebar{K''^{-i,-j}}\to\cur_Z\). The introduction of the sign \(\epsilon\), which satisfies \(\epsilon(i+j)=\epsilon(i)\epsilon(j)(-1)^{ij}\), ensures consistency between \(\Tp h\T\) and \(\Tp g\Tp f \T\).
\subsection{Complex Hodge modules}
\label{sec:org6cd7cae}
Consider as before a triple \(\T=(\Mt',\Mt'',\mathfrak{s})\) on \(X\). The Hermitian dual is the triple \(\T^{*}=(\Mt'',\Mt',\mathfrak{s}^{*})\), where \(\mathfrak{s}^{*}\) is the Hermitian adjoint pairing defined above. The Tate twist by \(l\in \mathbb{Z}\) is \(\T(l)=(\Mt'(l),\Mt''(-l),\mathfrak{s})\), where in general \(\Mt(l)=z^l\Mt\) (recall that \(\Mt\) is a module over \(\Dt= R_F\D_X=\bigoplus_p F_p\D_X\)).

A polarized Hodge module of weight \(w\) is the data of a holonomic object of \(\Dt\mathrm{-Triples}(X)\) and a morphism \(\mathrm{S}\colon\T\to\T^{*}(-w)\), the polarization, satisfying certain inductive conditions in terms of nearby and vanishing cycles (see\autocite[Definition 14.2.2]{Sabbah2018} for the full definition). In particular, \(\T\) is required to be strict, hence \(\Mt',\Mt''\) correspond to filtered \(\D_X\)-modules, and \(\mathrm{S}\) is an isomorphism of triples.

We will use the notation \((M,\mathrm{S})\) for a polarized Hodge module and, by strictness, consider \(M\) to be a triple \(((\M',F^{\bullet}\M'),(\M'',F^{\bullet}\M''),\mathcal{S})\). Note that \(S\) defines an isomorphism between \((\M',F^{\bullet}\M')\)  and \((\M'',F^{\bullet}\M'')(w)\). Let \((\M,F)=(\M',F)\); through this isomorphism, \((M,\mathrm{S})\) is isomorphic to a polarized Hodge module of the form \(((\M,F),(\M,F)(w),\mathcal{S})\) with polarization \((\id,\id)\), where \(\mathcal{S}\) is defined by composing \(\mathfrak{s}\) and \(\mathrm{S}\) appropriately. In this way, we will occasionally refer to \(\mathcal{S}\) as a polarization.

For a proper holomorphic map \(f\colon X\to Y\), let \(\Tp f {(M,\mathrm{S})}=(\Tp f M, \Tp f \mathrm{S})\) in the category of triples. The following now follows directly from the results of the previous section.
\begin{lemma}
\label{lemma:sign_compatibility_composition}
Let \(X\xrightarrow f Y \xrightarrow g Z\) be proper holomorphic maps, and \((M,\mathrm{S})\) a polarized Hodge module of weight \(w\) on \(X\). Then the complexes \(\Tp {(g\circ f)} {(M,\mathrm{S})}\) and \(\Tp g \Tp f {(M,\mathrm{S})}\) are quasi-isomorphic.
\end{lemma}
\subsubsection{A remark on polarizations in terms of perverse sheaves}
\label{sec:org1bcedfc}
The compatibility in \autoref{lemma:sign_compatibility_composition} is crucial for arriving at the correct signs for the polarizations in \autoref{lemma:P1-factorization} below, which is the technical heart of the new contributions of this paper. However, in the original formulation of the theory of Hodge modules by Saito \autocite{Saito1988}, where polarizations are defined in terms of perverse sheaves, this basic compatibility seems difficult to prove.

For the sake of the interested reader, we will point towards the source of the difficulties below. The reader that is only interested in the proof of \autoref{thm:decomp} can safely skip this section.

To explain the difficulty, recall that in constructing the pushforward of a sesquilinear pairing, a sign is introduced in the tensor product reordering isomorphism
\begin{align}
\M'\otimes\Sp(\D_X)\otimes \widebar{\M''\otimes\Sp(\D_X)}\xrightarrow\sim \M'\otimes \widebar{\M''}\otimes\Sp_{X,\widebar{X}\to Z,\widebar{Z}}(\D_{X,\widebar{X}})
\end{align}
when \(\M''\) is a complex of \(\D\)-modules, as described in the previous section. Consider the composition
\begin{align}
&f_{\ast}\left(\M'\otimes\Sp(\D_X)\right)\otimes f_{\ast}\left(\widebar{\M''\otimes\Sp(\D_X)}\right)\\
\to &{f_{\ast}}\left(\M'\otimes\Sp(\D_X)\otimes \widebar{\M''\otimes\Sp(\D_X)}\right)\nonumber\\
\xrightarrow\sim &f_{\ast}\left(\M'\otimes \widebar{\M''}\otimes\Sp_{X,\widebar{X}\to Z,\widebar{Z}}(\D_{X,\widebar{X}})\right)\nonumber
\end{align}
This represents a morphism \(\Dp f \M' \otimes \widebar{\Dp f \M''}\to \Dpb f (\M'\otimes \widebar{\M''})\).

In Saito's formulation, a polarization of a (in this case rational) Hodge module of weight \(w\) takes the form of a pairing
\begin{align}
S\colon K\otimes K\to a_X^{!}\mathbb{Q}(w),
\end{align}
where \(K\) is the \(\mathbb{Q}\)-perverse sheaf underlying the given Hodge module, and \(a_X\colon X\to \mathrm{pt}\) is the constant map. The pushforward of \(S\) along \(f\colon X\to Y\) as a pairing \(\sh{H}f_{\ast}S\colon\mathbf{R}f_{\ast}K\otimes \mathbf{R}f_{\ast}K\to a_Y^{!}\mathbb{Q}(w)\), in Saito's notation, is then defined as a composition (see \autocite[Lemme 5.2.9]{Saito1988}) 
\begin{align}
\mathbf{R}f_{\ast}K\otimes \mathbf{R}f_{\ast}K\to &\mathbf{R}f_{\ast}(K\otimes K) \\
\xrightarrow{\mathbf{R}f_{\ast}S} &\mathbf{R}f_{\ast}a_X^{!}\mathbb{Q}(w)\nonumber\\
\xrightarrow{ \operatorname{Tr_{X/Y}}} &a_Y^{!}\mathbb{Q}(w)\nonumber
\end{align}

To define the first map in this sequence, however, one must choose resolutions \(I(K)\) and \(I(K\otimes K)\) of \(K\) and \(K\otimes K\) respectively, adapted to the pushforward \(f_{\ast}\), and define a map \(I(K)\otimes I(K)\to I(K\otimes K)\), then define \(\mathbf{R}f_{\ast}K\otimes \mathbf{R}f_{\ast}K\to \mathbf{R}f_{\ast}(K\otimes K)\) as the induced map \(f_{\ast}I(K)\otimes f_{\ast}I(K)\to f_{\ast}I(K\otimes K)\). In the MHM project, Spencer resolutions are used for this purpose, as above.

Saito chooses to use the Godement resolution (constructed from sheaves of discontinuous sections of the underlying constructible sheaves) \autocite[Section 2.3, particularly 2.3.4 and 2.3.7]{Saito1988}. Assembling the Godement resolutions of the individual constructible complexes underlying \(K\) into a double complex, and taking the associated simple complex, gives a resolution of \(K\) as a complex of constructible sheaves. Applying the same procedure to \(K\otimes K\) gives the desired resolutions \(I(K)\) and \(I(K\otimes K)\), and one can indeed construct a map \(I(K)\otimes I(K)\to I(K\otimes K)\) as desired. However, following Deligne's sign rules for tensor products of complexes, as Saito does, one sees that this map must involve some signs that depend on the degrees of the terms of \(I(K)\) \emph{as a complex of constructible sheaves}.

Let us try to give an indication of where the issues lie. The following will necessarily be a sketch.

Suppose \(A\) and \(B\) are constructible sheaves. If \(S\) is some bilinear pairing between \(A\) and \(B\), we want to compare \(\sh H(g\circ f)_{\ast}S\) and \(\sh H g_{\ast}\sh H f_{\ast}S\). To do so, we must understand how the following maps compare:
\begin{align}\label{eq:tensor-1}
&\mathbf{R}(g\circ f)_{\ast} A \otimes \mathbf{R}(g\circ f)_{\ast} B\to\mathbf{R}(g\circ f)_{\ast}(A\otimes B)\\
&\mathbf{R}g_{\ast}\mathbf{R}f_{\ast} A\otimes\mathbf{R}g_{\ast}\mathbf{R}f_{\ast}B\to \mathbf{R}g_{\ast}(\mathbf{R}f_{\ast}A\otimes \mathbf{R}f_{\ast}B)\to \mathbf{R}g_{\ast}\mathbf{R}f_{\ast}(A\otimes B)\label{eq:tensor-2}
\end{align}

For any constructible sheaf \(C\), let \(G^{\bullet}C\) be the associated Godement resolution. To understand the maps in equations \eqref{eq:tensor-1} and \eqref{eq:tensor-2} in concrete terms, one ends up looking at the following diagram for various values of \(m\) and \(n\):
\begin{equation*}
\begin{tikzcd}[column sep = -5em]\label{diag:signs}
(g\circ f)_{\ast}G^{n}A\otimes (g\circ f)_{\ast} G^{n}B\ar[rr]&& (g\circ f)_{\ast}G^{m+n}(A\otimes B)\\
\bigoplus_{i+j=n}g_{\ast}G^{i}f_{\ast}G^jA\otimes\bigoplus_{k+l=m}g_{\ast}G^kf_{\ast}G^lB
\ar[u]\ar[dr]
&&\bigoplus_{\begin{subarray}{l} i+j=m\\ k+l=n\end{subarray}}g_{\ast}G^{i+k}f_{\ast}G^{j+l}(A\otimes B)\ar[u] \\
&\bigoplus_{\begin{subarray}{l} i+j=m\\ k+l=n\end{subarray}}g_{\ast}G^{i+k}(f_{\ast}G^jA\otimes f_{\ast}G^lB)\ar[ur]
\end{tikzcd}
\end{equation*}
Here the two vertical maps are the natural quasi-isomorphisms. The top horizontal map represents \eqref{eq:tensor-1}, while the composition across the bottom three terms represents \eqref{eq:tensor-2}.

Now in general, one can check that for complexes \(K, L\) of constructible sheaves, the natural map \(\phi\colon G^{\bullet}K\otimes G^{\bullet}L\to G^{\bullet}(K\otimes L)\) looks like \(\phi(\alpha\otimes\beta)=(-1)^{an}\alpha\otimes\beta\) for \(\alpha\in G^a K^m\) and \(\beta\in G^bL^n\), where on both sides we consider \(\alpha\) and \(\beta\) as discontinuous sections of the constructible sheaves making up \(K\) and \(L\) respectively.
On a local section
\begin{align}
\alpha\otimes\beta\in g_{\ast}G^{i}f_{\ast}G^jA\otimes g_{\ast}G^kf_{\ast}G^lB,
\end{align}
of the direct sum in the middle left of the diagram, a computation now shows that this diagram commutes up to a sign of \((-1)^{il}\).

Comparing the signs in the pairings \(\sh Hg_{\ast}\sh Hf_{\ast}S\) and \(\sh H (g\circ f)_{\ast}S\) on \(\mathbf{R}(g\circ f)_{\ast}K\) for a composition \(X\xrightarrow f Y \xrightarrow g Z\) and perverse sheaf \(K\) underlying a Hodge module, one now sees that there are certain signs involved that depend on the constructible degrees of the objects involved. These signs do not seem to be cancelled out by anything else, and so it is not the case that \(\sh Hg_{\ast}\sh Hf_{\ast}S=\sh H (g\circ f)_{\ast}S\) in general, even on terms of the Leray spectral sequence, and there does not seem to be any sign correction that one can insert on the level of Hodge modules to fix this, since such a correction would be in terms of the perverse t-structure, not the constructible one. Even if these ``constructible signs'' were cancelled out, one still needs to explain the discrepancy factor \((-1)^{ij}\) between the sign factors \(\epsilon(i+j)\) and \(\epsilon(i)\epsilon(j)\) that show up in the natural polarizations on \(^{p}R^{i+j} (g\circ f)_{\ast}K\) and \(^{p}R^{i}g_{\ast}{}^{p}R^jf_{\ast}K\) respectively. This sign, defined in terms of the perverse t-structure, does not seem to result from the tensor product maps \(\mathbf{R}f_{\ast}K\otimes \mathbf{R}f_{\ast}K\to \mathbf{R}f_{\ast}(K\otimes K)\) like their analogues in the MHM project approach to polarizations do.
\subsection{Hodge modules with integral structure}
\label{sec:org3b1eb0c}
As in \autocite[Definition 2.2]{Schnell2015a}, a \(\mathbb{C}\)-Hodge module with integral structure will consist of the data \(M=(M,K_{\mathbb{Z}},\alpha)\) where \(M\) is a polarizable \(\mathbb{C}\)-Hodge module with quasi-unipotent local monodromy, \(K_{\mathbb{Z}}\) is a constructible complex with integral coefficients, and \(\alpha\) is a comparison isomorphism \(\alpha\colon K_{\mathbb{Z}}\otimes_{\mathbb{Z}}\mathbb{C}\xrightarrow\sim \operatorname{DR}M\), where \(\operatorname{DR}M\) is the de Rham complex of the \(\mathscr D\)-module underlying \(M\). Hodge modules with rational structures are defined similarly, with a rational perverse sheaf \(K_{\mathbb{Q}}\) in place of \(K_{\mathbb{Z}}\).
Let \(\hmq X n\) be the category of Hodge modules with rational structures on \(X\) of weight \(n\). Let \(\hmz X n\) be the full subcategory of \(\hmq X n\) consisting of those Hodge modules \(M\) such that \(K_{\mathbb{Q}}\) is of the form \(K_{\mathbb{Z}}\otimes_{\mathbb{Z}}\mathbb{Q}\) for some integral constructible complex \(K_{\mathbb{Z}}\) (note that \(K_{\mathbb{Z}}\) is not part of the data!). We say that such an \(M\) \emph{admits} an integral structure.

Note that \(\mathbb{Z}\)-structures are preserved under the usual operations in the theory of Hodge modules: The functor \((-)\otimes_{\mathbb{Z}}\mathbb{Q}\colon D^b_c(X,\mathbb{Z})\to D^b_c(X,\mathbb{Q})\) commutes with the six-functor formalism for constructible complexes, and similarly for extending to complex coefficients. However, direct sum decompositions of rational structures will not necessarily be compatible with underlying integral structures on the nose, but only up to finite index.

For nearby and vanishing cycles, suppose \(f\colon X\to \Delta\) is a holomorphic function, submersive over the punctured disk \(\Delta^{\ast}\). Let \(\exp\colon \mathbb{H}\to \Delta\) be the composition of the universal covering of \(\Delta^{\ast}\) by the upper half plane with the inclusion \(\Delta^{\ast}\to \Delta\), \(k\colon \widetilde{X}\to X\) the pullback of \(\exp\) along \(f\), and \(i\colon X_0=f^{-1}(0)\to X\). Recall that for a constructible complex \(C\), the nearby cycles along \(f\) are defined as
\begin{align}
\psi_fC=i^{-1}\mathbf{R}k_{\ast}(k^{-1}C)
\end{align}
and the vanishing cycles \(\phi_fC\) as the cone of the canonical map \(i^{-1}C\to\psi_fC\). This construction goes through whether \(C\) has coefficients in \(\mathbb{Z}\) or \(\mathbb{Q}\), and the construction commutes with \((-)\otimes_{\mathbb{Z}}\mathbb{Q}\). One can check that if \(C\) is a \(\mathbb{Q}\)-complex which admits an integral structure, then the unipotent nearby and vanishing cycles \(\psi_{f,1}C\) and \(\phi_{f,1}C\) also admit integral structures. More precisely, the usual eigenspace decomposition over \(\mathbb{Q}\) induces a decomposition over \(\mathbb{Z}\) up to finite index (needed to clear the denominators of the rational numbers involved in the maps in the decomposition).

There's a notion of a perverse t-structure on \(D_c^b(X,\mathbb{Z})\) which maps to the usual one on \(D^b_c(X,\mathbb{Q})\) under tensor product \autocite[Section 3.3]{Beilinson1982}, so we can take perverse cohomology objects in \(D^b_c(X,\mathbb{Z})\). In particular, if \(C\in D^b_c(X,\mathbb{Q})\) admits an integral structure and \(f\colon X\to Y\) is a proper morphism, then each \({}^pR^if_{\ast}C\) admits an integral structure. Similarly, we can take intermediate extensions of local systems with integral coefficients, compatible with the usual intermediate extension functor for rational local systems.

We recall the structure theorem for Hodge modules, see \autocite[Theorem 16.2.1 (Structure theorem)]{Sabbah2018} and the S-decomposability property in \autocite[Theorem 14.2.17 (Main properties of polarizable Hodge modules)]{Sabbah2018}. Note that while the proof of the structure theorem uses the direct image theorem, it only does so in the specific context of blowups, which are projective morphisms. Hence we are free to use the structure theorem here.
\begin{theorem}
\label{thm:hm-structure}
Any polarizable pure Hodge module \(M\) admits a direct sum decomposition \(M=\bigoplus_Z M_Z\) where each \(M_Z\) is a polarizable pure Hodge module of the same weight as \(M\) which has strict support on an irreducible closed subvariety \(Z\subset X\). Note that the sum is taken over distinct subvarieties.

For each \(M_Z\), there exists a Zariski open subset \(U\subset Z\) such that \(M_Z|_U\) corresponds to a polarizable \(\mathbb{C}\)-VHS on \(U\), and \(M_Z\) is the intermediate extension of this VHS as a pure Hodge module.

Finally, \(M\) admits an integral structure if and only if each \(M_Z\) does, and a given \(M_Z\) admits an integral structure if and only if the corresponding VHS \(M_Z|_U\) does.
\end{theorem}
\begin{proof}
The statements about Hodge modules are due to \autocite{Sabbah2018}, as stated above. The statement comparing integral structures on \(M_Z\) and \(M_Z|_U\) is \autocite[Lemma 2.4]{Schnell2015a}.

If each \(M_Z\) admits an integral structure, then obviously so does \(M\). For the other direction, suppose \(M\) admits an integral structure, and fix one. We will first show that the decomposition \(M=\bigoplus_ZM_Z\) is defined over \(\mathbb{Q}\). Arguing inductively over the supports of \(M\) ordered by inclusion, it suffices to show that the decomposition \(M=M_Z\oplus N\) is defined over \(\mathbb{Q}\) in the case where \(Z\) is a maximal support, i.e. not strictly contained in any other support. Let \(K,K_Z\) be the complex perverse sheaves associated to \(M\) and \(M_Z\). There is a dense open subset \(U\subseteq Z\) such that \(\sh H^{-\dim Z}K|_U\) (the constructible cohomology sheaf in degree \(-\dim Z\)) and \(K_Z|_U[-\dim Z]\) are both local systems, and by maximality of \(Z\) among the supports, they are equal. It follows that \(K_Z|_U\) and its intermediate extension \(K_Z\) naturally have a rational structure induced from that of \(K\). Note that the splitting maps \(M_Z\to M\to M_Z\) are defined over \(\mathbb{Q}\), since they are entirely determined by their restrictions to \(U\), so \(N\) also has an induced rational structure.

To compare integral structures on \(M\) and the summands \(M_Z\), first observe that if \(\rho\colon G\to \operatorname{GL}(A)\) is a representation of a group \(G\) on a finitely generated abelian group \(A\), and if \(\rho_{\mathbb{Q}}\colon G\to \operatorname{GL}(A_{\mathbb{Q}})\) preserves a rational subspace \(V\subset A_{\mathbb{Q}}\), then \(\rho\) preserves the subgroup \(A\cap V\) of \(A\) (interpreted as the preimage of \(V\) under \(A\to A_{\mathbb{Q}}\)). That is, any rational subrepresentation of \(\rho_{\mathbb{Q}}\) admits an integral structure.

Now let \(K_{\mathbb{Q}}\) be the rational structure on \(M\). Given \(Z\) among the supports of \(M\), let \(K_{Z,\mathbb{Q}}\) be the rational structure on \(M_Z\), and take a non-empty Zariski open subset \(U\subset Z\) such that \(K_{Z,\mathbb{Q}}|_U[-\dim Z]\)  and \(\sh H^{-\dim Z}K_{\mathbb{Q}}|_U\) are local systems, where \(\sh H^{-\dim Z}\) is taken using the standard t-structure on constructible sheaves. Then \(K_{Z,\mathbb{Q}}|_U[-\dim Z]\) is a local subsystem of \(\sh H^{-\dim Z}K_{\mathbb{Q}}|_U\), and the latter admits an integral structure by assumption. Thus \(K_{Z,\mathbb{Q}}\) admits an integral structure as desired.
\end{proof}
\subsection{A compactification theorem by Sommese}
\label{28cefe6b-e51a-4080-8d90-7bd1c41cee45}
\begin{definition}
A proper modification of a complex analytic space \(X\) is a proper morphism \(f\colon X'\to X\) such that there exists a dense, Zariski open subset \(U\subset X\) such that \(f^{-1}(U)\) is dense and \(f|_{f^{-1}(U)}\colon f^{-1}(U)\to U\) is an isomorphism.
\end{definition}
Note that a proper modification is also a bimeromorphism, that is, the inverse \((f|_{f^{-1}(U)})^{-1}\colon U\to f^{-1}(U)\) is also meromorphic \autocite[Remark 1.8]{Peternell1994}.

\begin{definition}
A compact complex space \(X\) is in Fujiki's class \(\mathscr C\) if there exists a surjective holomorphic map \(f\colon Y\to X\) with \(Y\) a compact Kähler manifold.
\end{definition}
Fujiki introduced and studied this class of spaces in the papers \autocite{Fujiki1978,Fujiki1978/79,Fujiki1984}. We recall from Fujiki's results that the class \(\mathscr C\) is closed under taking subspaces and meromorphic images, and that any irreducible component of the Douady space of compact subvarieties of a space in class \(\mathscr C\) is again a space in class \(\mathscr C\). By
\autocite[Théorème 3]{Varouchas1986} and resolution of singularities, if \(X\) is in class \(\mathscr C\) then there exists a proper modification \(f\colon X'\to X\) where \(X'\) is a compact Kähler manifold. In fact by \autocite[Theorem 5]{Varouchas1989} we can even take \(f\) to be projective.

\begin{definition}
Let \(X\) be a complex analytic space. By a compactification of \(X\), we mean an open embedding \(X\subset \widebar{X}\) of \(X\) in a compact complex analytic space\(\widebar{X}\) such that \(X\) contains a non-empty Zariski open subset of \(\widebar{X}\).
\end{definition}

The following theorem goes back to Sommese \autocite[Proposition III, Remark III-C]{Sommese1978} and is also stated in a similar form in \autocite[Theorem 7.7]{Bakker2023}. However, parts of Sommese's argument seem problematic, or at least incomplete, even in the projective case. We give a slightly modified proof here, relying on Hironaka's flattening theorem to sidestep the problematic parts of Sommese's argument.

\begin{theorem}
\label{thm:sommese-compactification}
Let \(X\) be a Zariski open subset of an irreducible compact complex analytic space \(\widebar{X}\) in Fujiki's class \(\mathscr C\), \(Y\) a reduced complex analytic space, and \(\pi\colon X\to Y\) a surjective proper holomorphic map with connected fibres. Then there exists a commutative diagram
\begin{equation*}
\begin{tikzcd}
X\ar[d,"\pi"]&\ar[l]X'\ar[d]\ar[r,hook]&\widebar{X}'\ar[d,"\pi'"]\\
Y&\ar[l]Y'\ar[r,hook]&\widebar{Y}'
\end{tikzcd}
\end{equation*}
where the vertical maps are surjective proper maps with connected fibres, the left side horizontal maps are proper modifications, the right side horizontal maps are compactifications where \(\widebar{X}'\) and \(\widebar{Y}'\) are compact Kähler manifolds, and \(\widebar{X}'\) and \(\widebar{X}\) are bimeromorphic.
\end{theorem}

\begin{proof}
By replacing \(Y\) and \(\widebar{X}\) with proper modifications, we can arrange that \(\pi\) is flat by Hironaka's flattening theorem \autocite[Corollary 1]{Hironaka1975}, while still having connected fibres. Note that Fujiki's class \(\mathscr C\) is closed under proper modifications. Abusing notation, we will thus assume from the start that \(\pi\) is flat.

Now consider the Douady space of \(\widebar{X}\). By functoriality and flatness of \(\pi\), \(Y\) maps to the Douady space; let \(C\) be an irreducible component containing the image of \(Y\). By \autocite{Fujiki1984}, \(C\) is compact and in class \(\mathscr{C}\). Since the fibres of \(\pi\) are connected, our assumptions on \(X\) and \(Y\) imply that the general fibre is irreducible, so by \autoref{thm:rigidity-irreducible} below applied to \(\pi\), the functorial map \(j\colon Y\to C\) is an open embedding over a non-empty Zariski open subset of \(Y\).

Let \(d\colon Z\to C\) be the family of subvarieties of \(\widebar{X}\) parametrized by \(C\), and \(\phi\colon Z\to \widebar{X}\) the projection to \(\widebar{X}\). Let \(U\subset Y\) be the Zariski open subset over which the fibres of \(\pi\) are irreducible, and let \(Z_U=\phi^{-1}(\pi^{-1}(U))\).

If for \(c\in C\) the fibre \(\delta^{-1}(c)\) intersects \(Z_U\), then \(\phi(\delta^{-1}(c))\) intersects an irreducible fibre of \(\pi\) in a Zariski open subset, hence contains that entire fibre. Since the fibre and \(\phi(\delta^{-1}(c))\) are flat deformations of each other, they must coincide, so \(Z_U=d^{-1}(j(U))\). In particular, applying Remmert's proper mapping theorem to the complement of \(Z_U\) in \(Z\), \(j(U)\) is Zariski open in \(C\). This implies that \(\phi\colon Z\to \widebar{X}\) is an isomorphism when restricted to \(Z_U\), hence is a proper modification. This also shows that \(C\) is a compactification of \(Y\).

Finally, \(Z\) and \(C\) are compact in Fujiki class \(\mathscr C\). Hence there is a proper modification \(\pi'\colon \widebar{X}'\to \widebar{Y}'\) of \(d\colon Z\to C\) such that \(\widebar{X}'\) and \(\widebar{Y}'\) are compact Kähler manifolds as desired. In particular, we have also constructed bimeromorphic morphisms \(\widebar{X}'\to Z\to \widebar{X}\).
\end{proof}

The following rigidity theorem is given, with ``connected'' in place of ``irreducible'', as \autocite[Sub-Lemma A]{Sommese1978} as part of the proof of \autocite[Proposition III]{Sommese1978}. However, the statement given by Sommese is false as written; if \(F\) is a reducible but connected pure-dimensional fibre, take \(F'\) to be a component of \(F\) to get a contradiction. This difference between ``connected'' and ``irreducible'' does not matter for \autoref{thm:sommese-compactification} or for \autocite[Proposition III]{Sommese1978}, however.
\begin{theorem}
\label{thm:rigidity-irreducible}
Let \(p\colon X\to Y\) be a proper surjective map between complex analytic spaces. Assume \(p\) has irreducible fibres. Let \(F\) be a fibre. Then there exists a neighbourhood \(U\) of \(F\) such that any compact irreducible analytic space \(F'\subset U\) with \(\dim F'=\dim F\) is a fibre of \(p\).
\end{theorem}
\begin{proof}
For \(F'\) as in the statement, \(p(F')\) is a subvariety of \(Y\) by Remmert's proper mapping theorem \autocite[§10.6]{Grauert1984}, but if \(U\) is sufficiently small, \(p(F')\) belongs to a Stein open neighbourhood of \(p(F)\) and hence must be a point. By upper semi-continuity of dimension, the fibres of \(p\) near \(F\) have dimension at most \(\dim F\), so by the irreducibility assumptions, \(F'\) is an entire fibre.
\end{proof}

Let us briefly recall the following construction in the complex analytic setting. We will use this often to reduce from working with an arbitrary generically defined VHS, to working with a VHS which is defined on the complement of a simple normal crossings divisor and has unipotent local monodromy.
\begin{lemma}
\label{lemma:construct-analytic-covers}
Let \(Z\) be a compact, closed, irreducible analytic subset of a Kähler manifold \(X\), and \(U\subseteq Z\) a smooth Zariski-open subset.

For any étale cover \(U'\to U\), there exists a compactification \(Z'\supseteq U'\) by a compact Kähler manifold and an extension \(\colon Z'\to Z\) of the covering map, such that \(Z'-U'\) is a simple normal crossings divisor.
\end{lemma}
\begin{proof}
We can construct \(Z'\) by first applying a theorem of Grauert and Remmert \autocite[Théorème 5.4]{SGA1} to compactify \(U\), and then taking log resolutions. Write the resulting map \(\pi\colon Z'\to Z\) as a composition \(Z'\to \widehat{Z}\to Z\). Again, the sum of exceptional divisors of the blowups in the construction of \(Z'\to \widehat{Z}\) is \(\pi\)-anti-ample, so \(Z'\) is also Kähler.
\end{proof}

\begin{corollary}
\label{cor:VHS-implies-algebraicity}
Let \(X\) be a compact Kähler manifold, and suppose \(U\subset X\) is a Zariski-open subset with a polarizable \(\mathbb{C}\)-VHS \(V\) which admits an integral structure, and whose period map has rank \(k\) at some point in \(U\). Then the transcendence degree of the field of meromorphic functions on \(X\) is at least \(k\).
\end{corollary}
\begin{proof}
The assumption and conclusion are both invariant under bimeromorphic modifications and finite covers, so by \autoref{lemma:construct-analytic-covers}, we can assume that \(U\) is the complement of a simple normal crossings divisor and that \(V\) has unipotent local monodromy, and is extended maximally.

Then its period map \(\Phi\colon U\to \Gamma\backslash D\) is proper \autocite[Section 9]{Griffiths1970}, so we can take its Stein factorization \(U\to Y\to \Phi(U)\) where the first map has connected fibres and the second is finite. Now apply \autoref{thm:sommese-compactification} to \(U\to Y\) to extend to \(\widebar{\Phi}\colon X\to \widebar{Y}\), after possibly modifying \(X\) further. Then \(Y\) is finite over the period image \(\Phi(U)\), hence it admits a VHS whose period map immersive, so by \autocite[Corollary 1.3]{Brunebarbe2020}, \(\widebar{Y}\) is projective. Since \(\widebar{\Phi}\colon X\to \widebar{Y}\) is surjective and \(\widebar{Y}\) has dimension at least \(k\), we are done.
\end{proof}

Note that without the use of \autoref{thm:sommese-compactification}, we could not guarantee that the map \(U\to \widebar{Y}\) gives a meromorphic map on \(X\), since the map on \(U\) might have essential singularities on the boundary.

\begin{remark}
Instead of the analytic result of \autocite{Brunebarbe2020}, one could also use the resolution of the Griffith's conjecture on the quasi-projectivity of period images by Bakker-Brunebarbe-Tsimerman, specifically the factorization result \autocite[Theorem 7.6]{Bakker2023}. The method of Brunebarbe-Cadorel is simpler, though, and suffices for our purpose.
\end{remark}
\section{The decomposition theorem}
\label{aa176a37-ee8b-405b-beda-6b0568decba7}
In this section, we prove \autoref{thm:decomp}. Most of Saito's argument in \autocite{Saito1990a} is independent of \autocite{Kashiwara1986}, and indeed is essentially the same as the original argument in the projective case in \autocite{Saito1988}, and for complex Hodge modules in \autocite{Sabbah2018}. We will follow the same overall strategy as in those papers, arguing by induction on the dimension of the support \(Z\) of \(M\). By \autocite[Section 14.4]{Sabbah2018} (or \autocite[Proposition 5.3.4]{Saito1988} in the original setting), reducing \(\dim f(Z)\) inductively via repeated use of nearby and vanishing cycles, it suffices to prove the decomposition theorem for Hodge modules supported on fibres of \(Z\) over \(Y\), so we will from the start assume that \(\dim f(Z)=0\) (note that in \autocite{Saito1988,Saito1990a}, the projectivity assumption, resp. the results of \autocite{Kashiwara1986}, are only used after this reduction).

Furthermore, by the induction, we can assume the results of the decomposition theorem for any Hodge module supported on a subvariety \(W\subset X'\) with respect to any proper, relatively Kähler morphism \(g\colon X'\to Y'\) such that the fibres of \(W\) over \(Y'\) have dimension strictly less than \(\dim Z\). Likewise, we may use the decomposition theorem for projective morphisms.

By \autoref{thm:hm-structure}, it suffices to consider the case that the Hodge module \(M\) has strict support on \(Z\), which we will assume for the rest of the proof. Let \(V=M|_U\) denote the corresponding VHS on a smooth Zariski open subset \(U\subset Z\), and let \(D=Z-U\).
\subsection{Behaviour under alterations}
\label{sec:orgacd3db2}

\begin{lemma}
\label{lemma:alteration}
Suppose we are in the situation above with \(\dim f(Z)=0\), and that \(\pi\colon Z'\to Z\) is a composition \(Z'\to \widehat{Z}\to Z\), where \(\widehat{Z}\to Z\) is finite, \(Z'\) is a compact Kähler manfold  and \(Z'\to \widehat{Z}\) is a composition of blowups in smooth subvarieties of \(\widehat{Z}\).
Let \(M'\in \hmz {Z'} n\) be the generic pullback of \(M\), i.e. the Hodge module corresponding to the generic VHS \((\pi|_{\pi^{-1}(U)})^{\ast}V\) by \autoref{thm:hm-structure}, and suppose that the conclusions of \autoref{thm:decomp} are satisfied for \(\Tp {a_{Z'}}M'\). Suppose finally that \(\Tp {a_{X}}M\) satisfies \autoref{thm:decomp}(3).

Then \(\Tp {a_{X}}M\) satisfies the conclusions of \autoref{thm:decomp}.
\end{lemma}
Note that if \(Z\) were smooth, we could forget about the ambient manifold \(X\) and just prove \autoref{thm:decomp} for \(\Tp {a_{Z}}M\). If \(Z\) is singular, however, it is convenient to have an ambient manifold to work with Kähler forms. Since we take \(Z'\) to be smooth in the lemma, we do not need a bigger ambient manifold there, even though one is often available. Thus this lemma, together with \autoref{lemma:construct-analytic-covers}, allows us to reduce to the case where \(V\) is a VHS on the complement of a normal crossings divisor. By allowing the finite map \(\widehat{Z}\to Z\), we can ensure that the local monodromy of \(V\) is unipotent.

\begin{proof}
First note that \(M\) is a direct summand of \(\Tp {\pi}M'\) under a decomposition of \(\Tp {\pi}M'\) on \(X\). Thus \(\Tpc {a_{X}} kM\) is a direct summand of \(\Tpc {a_{Z'}} kM'\) for every \(k\), so the former is strict, and the cohomology of \(M\) consists of polarizable Hodge structures of the expected weight.

Let \(E\) be the sum of the exceptional divisors in the sequence of blowups \(Z'\to \widehat{Z}\). Then \(\mathcal{O}_{Z'}(-E)\) is \(\pi\)-ample; write \(\eta=c_1(\mathcal{O}_{Z'}(-E))\). If \(\omega\) is a Kähler class on \(X\), then \(\pi^{\ast}\omega+c\eta\) is a Kähler class on \(Z'\) for \(c>0\) sufficiently small, say \(0<c\le c_0\).

For the sake of notation, let \(H'^j=\Tpc {a_{Z'}} jM'\) and \(H^j=\Tp {a_{Z}}M\), and let \(\omega\colon H^j\to H^j(1)\) and \(\omega,\eta\colon H'^j\to H'^j(1)\) denote the Lefschetz operators induced by \(\omega,\pi^{\ast}\omega\) and \(\eta\) respectively. Since the actions of \(\pi^{\ast}\omega\) and \(\omega\) agree on the summand \(M\) of \(\Tp {\pi}M'\), this abuse of notation should not lead to any ambiguity.

Choose now a polarization \(S\) of \(M\). This corresponds to a polarization of the generic VHS corresponding to \(M\) under \autoref{thm:hm-structure}, which we can pull back to get a polarization \(S'\) of \(M'\). Abusing notation, write \(S\) and \(S'\) also for the induced pairings on \(H\) and \(H'\). Since \(\Tp {a_{Z'}}M'\) satisfies the conclusions of \autoref{thm:decomp}, the triple
\begin{align}
\left(\bigoplus_j H'^j,c^{-1}\omega +\eta,S'\right)
\end{align}
is a polarized Hodge-Lefschetz structure for any \(0<c\le c_0\) (keeping in mind the abuse of terminology described in \autoref{rmk:lefschetz}), hence \(\bigoplus H'^j\) is a mixed Hodge structure polarized by \(S', C(\omega,c_0^{-1}\omega +\eta)\).

By \autoref{lemma:mhs-polarization-reduction-of-variables}, passing to the associated graded of the weight filtration \(W(\omega)\) yields the polarized bi-\(\mathfrak{sl}_2\)-Hodge structure
\begin{align}
\left(\bigoplus_{i,j} \operatorname{Gr}^{W(\omega)}_i H'^j,\omega, \operatorname{Gr}^{W(\omega)}\eta,S'\right),
\end{align}

Now, since \(M\) is a direct summand of \(\Tp {\pi}M'\) and the action of \(\pi^{\ast}\omega\) on \(\Tp {a_{Z'}}M'\) is compatible with that of \(\omega\) on \(\Tp {a_{X}}M\), \(\bigoplus \operatorname{Gr}_i^{W(\omega)}H^j\) is a direct summand of \(\bigoplus \operatorname{Gr}^{W(\omega)}_i H'^j\). Since \(\Tp {a_{X}}M\) satisfies hard Lefschetz, we have
\begin{align}
W_i(\omega) H^j= \begin{cases}
H^j & j\ge -i, \\
0&\mathrm{else},
\end{cases}\\
\operatorname{Gr}_i^{W(\omega)}\bigg(\bigoplus_jH^j\bigg)=H^{-i}.
\end{align}
Further, \(\operatorname{Gr}^{W(\omega)}\eta\) is just the operator induced by the morphisms \(\eta\colon \Tpc {\pi} j M'\to \Tpc {\pi} {j+2}M'(1)\) by \autoref{thm:decomp} applied to \(\pi\). Looking at supports, we see that this morphism vanishes on \(M\), hence \(H^j\) is in the \(\operatorname{Gr}^{W(\omega)}\eta\)-primitive part of \(\bigoplus_{i,j}\operatorname{Gr}_i^{W(\omega)}H'^j\).

Now \(\Tp {\pi}S'\) and \(S\) agree on the open locus \(U\). Since \(M\) has strict support, a pairing on \(M\) is uniquely determined by its restriction to \(U\), so we get the desired polarization on \(\Tp {a_{X}}M\).
\end{proof}
\subsection{Normal crossings case}
\label{sec:org4fc3d1e}
Let us first prove the decomposition theorem in the case where \(Z\) is smooth and \(D\) is a simple normal crossings divisor. We can then assume that \(Z=X\) for notational simplicity.

In general, \(\Tp {a_{X}}M\) computes the intersection cohomology of the VHS \(V\). Since \(D\) has only simple normal crossings, this can be computed using \(L^2\)-cohomology \autocite{Cattani1987,Kashiwara1987}, and since the Kähler identities (and hence the whole Kähler package) hold in that setting, \(\Tp {a_{X}}M\) satisfies hard Lefschetz \autocite[Theorem 6.4.2]{Kashiwara1987}.

By \autoref{lemma:alteration}, it suffices to prove the decomposition theorem for \(a_{X'+}M'\) where \(\pi\colon X'\to X\) is a composition of a finite morphism and a sequence of blowups in smooth loci, and \(M'\) is the generic pullback of \(M\). By \autoref{lemma:construct-analytic-covers} and since \(V\) has quasi-unipotent local monodromy, we can thus from the start reduce to the case where \(V\) actually has unipotent monodromy around the branches of \(D\), which we will assume going forward.

Let \(\Phi\) be the period map for \(V\) extended maximally on \(X\), with domain \(U\subseteq X\); let \(Y\) be the image of \(\Phi\) and consider \(\Phi\colon U\to Y\). Then \(\Phi\) is proper \autocite[Section 9]{Griffiths1970}. Note that if \(\dim X\le 1\) we are done essentially by the work of Zucker \autocite{Zucker1979} as in \autocite{Sabbah2018}, so assume \(\dim X\ge 2\).

Suppose that \(\Phi\) is constant. This is equivalent to saying that the monodromy of \(V\) preserves the Hodge structure at a reference point, and since the monodromy preserves the polarization of \(V\) as well, the monodromy is unitary. But the monodromy is also defined over the integers, hence every element of the monodromy group has finite order by Kronecker's theorem (stating that an algebraic integer all of whose Galois conjugates have complex modulus \(1\) must be a root of unity). As the monodromy group is a finitely generated linear group, the Jordan-Schur theorem \autocite[Theorem 36.2]{Curtis1962} now implies that it is finite. Hence by \autoref{lemma:construct-analytic-covers}, there is a finite cover on which \(V\) becomes the trivial VHS \(\mathbb{Z}_U\otimes H\) for a fixed integral, \(\mathbb{C}\)-polarizable Hodge structure \(H\). But then the intermediate extension \(M\) of this is the trivial VHS \(\mathbb{Z}_X\otimes H\), considered as a Hodge module, so \(\Tpc {a_{X}} kM=H^{k+\dim X}(X,\mathbb{\mathbb{Z}})\otimes H\) (up to torsion on the level of \(\mathbb{Z}\)-coefficients), and we are done by classical Hodge theory.

If \(\Phi\) is non-constant, then \autoref{cor:VHS-implies-algebraicity} implies that \(X\) admits a meromorphic function. Using \autoref{lemma:alteration}, we can reduce to the situation where \(X\) admits a surjection to \(\mathbb{P}^1\) by resolving the indeterminancy locus, ensuring that \(D\) remains normal crossings. Then \autoref{lemma:P1-factorization} below applies.

\begin{remark}
\label{rmk:projective}
If \(X\) is projective (which in particular happens if \(\Phi\) is immersive at some point, by \autocite[Corollary 1.3]{Brunebarbe2020}), this argument with period maps can be replaced by mapping a blowup of \(X\) to \(\mathbb{P}^1\) by a pencil of hyperplanes, then applying \autoref{lemma:P1-factorization}. Since the discussion above is the only part of our proof that uses the integral structure, we can thus deduce \autoref{thm:decomp} for projective morphisms and \(\mathbb{C}\)-Hodge modules that do not necessarily admit an integral structure.

Note, however, that even in the projective case we have to go through \autoref{lemma:P1-factorization} rather than Saito's original argument, since our Kähler class might not be in the ample cone of \(X\).
\end{remark}

\begin{lemma}
\label{lemma:P1-factorization}
Suppose \(f\colon X\to \mathbb{P}^1\) is a surjection from a compact Kähler manifold, and suppose \autoref{thm:decomp} holds for \(f\). Suppose also that \(M\in \operatorname{HM}_{\mathbb{Z}}(X,n)^p\) and that \(\Tp {a_{X}}M\) satisfies \autoref{thm:decomp}(3). Then \(\Tp {a_{X}}M\) satisfies the conclusions of \autoref{thm:decomp}.
\end{lemma}
\begin{proof}
Using \autoref{thm:decomp} for \(f\) and Deligne's theorem \autocite{Deligne1968}, take a decomposition \(\Tp fM\simeq\bigoplus_i\Tpc f iM[-i]\), say in the derived category of the underlying \(\Dt_{\mathbb{P}^1}\)-triples. Each summand is a polarizable Hodge module, and while the decomposition is not necessarily compatible with the integral structures, each summand on the right does admit an integral structure. Thus \autoref{thm:decomp} for the constant map on \(\mathbb{P}^1\) applies by the induction on dimension, and we get
\begin{align}
\Tpc {a_{X}} kM\cong\bigoplus_{i+j=k}\Tpc {a_{\mathbb{P}^1}} j\Tpc f iM
\end{align}
As each summand on the right is a polarizable Hodge structure, so is \(\Tpc {a_{X}} kM\).

It remains to check the polarization formula for \(\Tp {a_{X}}M\) with respect to a Kähler class \(\eta\) on \(X\), that is, to check \autoref{thm:decomp}(4). Consider first that \(\eta\) induces a map \(\eta\colon \Tp fM\to \Tp fM(1)[2]\) in the derived category of triples. Given a choice of decomposition of \(\Tp fM\), \(\eta\) decomposes as a sum of morphisms
\begin{align}
\Tpc f i M\to \Tpc f {i-l+2} M(1)[2-l]
\end{align}
in the derived category for various \(i\) and \(l\), i.e. classes in \(\ext^l(\Tpc f iM, \Tpc f {i-l+2}M(1))\). We can collect these classes according to how much they shift the cohomological degree to get elements
\begin{align}
\eta_{-l}\in\bigoplus_i \ext^l(\Tpc f iM,\Tpc f {i-l+2}M(1)).
\end{align}
For notational convenience, write \(H^k=\Tpc {a_{X}} kM\) and \(H_{i,j}=\Tpc {a_{\mathbb{P}^1}} j\Tpc f iM\); then the classes \(\eta_{-l}\) induce operators
\begin{align}
\eta_{-l}\colon H_{i,j}\to H_{i-l+2,j+l}(1).
\end{align}
But as we are on \(\mathbb{P}^1\), the operators \(\eta_{-l}\) vanish for \(l\ge 3\). Further, there's a unique choice of decomposition of \(\Tp fM\) such that \(\eta_{-1}=0\) and \(\eta_0\) and \(\eta_{-2}\) commute \autocite[Proposition 3.5]{Deligne1994}. We will work with this decomposition throughout.

Let \(L\) be the pullback along \(f\) of a Kähler class on \(\mathbb{P}^1\). This acts on the cohomology groups by \(L\colon H_{i,j}\to H_{i, j+2}(1)\). The two non-zero components of \(\eta\) are \(\eta_0\colon H_{i,j}\to H_{i+2,j}(1)\) and \(\eta_{-2}\colon H_{i,j}\to H_{i,j+2}(1)\). Since \(L\) and \(\eta\) commute, we see using the bigrading that \(L\) commutes with \(\eta_0\) and \(\eta_{-2}\) as well.

Now choose a polarization \(S\) of \(M\). Note that \(\eta_0\) acting on \(H_{i,j}\) is induced by the relative Lefschetz operator \(\Tpc f iM\to \Tpc f {i+2}M(1)\) induced by \(\eta\). Hence relative hard Lefschetz for \(\eta\) and hard Lefschetz for \(L\) imply that \(\left(\bigoplus_{i,j}H_{i,j},\eta_0,L\right)\) is a bi-\(\mathfrak{sl}_2\)-Hodge structure (with the Lefschetz operators being the weight-increasing operators), and \autoref{thm:decomp} for \(f\) and for the constant map on \(\mathbb{P}^1\) imply that this structure is polarized by
\begin{align}
S':=\bigoplus_{i,j\ge 0}\Tp {a_{\mathbb{P}^1}}\Tp fS|_{H_{i,j}\otimes H_{-i,-j}}.
\end{align}
Concretely, this means that on the \((\eta_0,L)\)-bi-primitive part of \(H_{-i,-j}\), the pairing
\begin{align}
(x,y)\mapsto S'(\eta_0^iL^jx, \widebar{C_{D}y})
\end{align}
is a positive-definite hermitian form. We need to show that
\begin{align}
S''=\bigoplus_{k\ge 0} \Tp {a_{X}}S|_{H^k\otimes H^{-k}}
\end{align}
polarizes the \(\mathfrak{sl}_2\)-Hodge structure \(\left(\bigoplus_k H^k,\eta_0+\eta_{-2}\right)\). By \autoref{lemma:bisl2}, it suffices to show that \(S''\) polarizes the bi-\(\mathfrak{sl}_2\)-Hodge structure
\begin{align}
\left(\bigoplus_{i,j}H_{i,j},\eta_0,\eta_{-2}\right).\label{eq:goal-sl2}
\end{align}
By \autoref{lemma:sign_compatibility_composition}, \(S'\) and \(S''\) agree on each of the spaces \(H_{i,j}\otimes H_{-i,-j}\), so it suffices to show that \(S'\) polarizes the structure \eqref{eq:goal-sl2}.

Note that \(H_{i,j}\) can only be non-zero for \(j=-1,0,1\), so \(\eta_{-2}\) and \(L\) can only act non-trivially on the spaces \(H_{i,-1}\). On such a space, either operator is an isomorphism \(H_{i,-1}\xrightarrow\sim H_{i,1}(1)\). We will show that the automorphism \(L^{-1}\eta_{-2}\) of \(H_{i,-1}\) only has positive real eigenvalues.

To do this, first recall that \(\eta_{-2}\) and \(L\) commute with \(\eta_0\), hence in particular preserve the \(\eta_0\)-primitive decomposition of \(H_{-i,-1}\), so it suffices to look at the restricted isomorphisms
\begin{align}
P_{\eta_0}H_{-i,-1}\xrightarrow{\eta_{-2},L}P_{\eta_0}H_{-i,1}(1)
\end{align}
and the eigenvalues of \(L^{-1}\eta_{-2}\) on \(P_{\eta_0}H_{-i,-1}\).

Suppose \(\lambda\) is an eigenvalue and \(s\in P_{\eta_0}(H_{-i,-1})_{\mathbb{C}}\) a corresponding eigenvector, such that \(\eta_{-2}s=\lambda Ls\). Note that \(s\) is automatically \(\eta_{-2}\) and \(L\)-primitive. Then
\begin{align}
S'(\eta_0^i\eta_{-2}s, \widebar{C_Ds})=\lambda S'(\eta_0^i Ls, \widebar{C_Ds}).
\end{align}
The former is the self-pairing of \(s\) under a hermitian form, hence a real number, and the latter is a (positive real) norm times \(\lambda\), hence \(\lambda\) is real. Suppose \(\lambda\) is negative. Then replacing \(\eta\) with \(\eta-\lambda L\), which is still a Kähler form on \(X\), we see that \((\eta-\lambda L)^{i+1}s=\eta_0^{i+1}s=0\) since \(s\) is primitive, contradicting hard Lefschetz for \(\eta-\lambda L\).

It follows that for every \(a\ge 0\), the map \(\eta_{-2}+aL\colon H_{-i,-1}\to H_{-i,1}(1)\) is an isomorphism. In particular, the bi-primitive spaces of the bi-\(\mathfrak{sl}_2\)-Hodge structures
\begin{align}
&\left(\bigoplus_{i,j}H_{i,j},\eta_0,L\right)\label{eq:L-bi-sl2}\\
&\left(\bigoplus_{i,j}H_{i,j},\eta_0,\eta_{-2}+aL\right)\label{eq:aL-bi-sl2}
\end{align}
are the same for every \(a\ge 0\).

We now claim that since \(S'\) polarizes the bi-\(\mathfrak{sl}_2\)-Hodge structure \eqref{eq:L-bi-sl2}, \(S'\) also polarizes the structure \eqref{eq:aL-bi-sl2} for \(a\gg 0\). Since the bi-primitive subspaces of the structures are the same for all \(a\ge 0\), the claim is just that the hermitian pairings
\begin{align}
(x,y)\mapsto S'(\eta_0^i (\eta_{-2}+aL)x, \widebar{C_Dy})\label{eq:aL-pairings}
\end{align}
induced by \(S\) on the bi-primitive subspaces are positive definite for sufficiently large \(a\). But up to rescaling by the positive factor \(a^{-1}\), this tends to the positive definite pairing
\begin{align}
(x,y)\mapsto S'(\eta_0^i Lx, \widebar{C_Dy})
\end{align}
in the limit.

The pairings \eqref{eq:aL-pairings} furthermore remain non-degenerate for every \(a\ge 0\) by the hard Lefschetz property, and in a connected 1-dimensional family of non-degenerate hermitian pairings, one of them is positive definite if and only if they all are. Taking \(a=0\), we conclude that \(\left(\bigoplus_{i,j} H_{i,j},\eta_0,\eta_{-2}\right)\) is polarized by \(S'\).
\end{proof}

\begin{remark}
In the proof of \autoref{lemma:P1-factorization}, it was relatively easy to show that \(\Tpc {a_{X}} kM\) is a polarizable Hodge structure. Unfortunately, this does not suffice for the inductive step of the proof of the whole direct image theorem, where we really do need a specific polarization formula as in \autoref{thm:decomp}(4) in order to ensure that the polarization obtained pointwise over the base of a morphism \(f\colon X\to Y\) via nearby and vanishing cycles gives a global polarization on \(Y\).

Besides the hard Lefschetz theorem, establishing the polarization formula is the main difficulty in our proof of \autoref{thm:decomp}.
\end{remark}
\subsection{Reduction to normal crossings case}
\label{sec:org703228e}
By \autoref{lemma:alteration} and \autoref{sec:org4fc3d1e}, applied to a log resolution of \((Z,D)\), it only remains to prove \autoref{thm:decomp}(3) (hard Lefschetz) for \(\Tp {a_{X}}M\) in the general case.

We can apply the reduction argument of \autocite[Section 6.3]{Mochizuki2022} as follows. Take an embedded log resolution \(\pi\colon X'\to X\) of \((Z,D)\). This is a sequence of blowups in smooth (proper) subvarieties of \(Z\) such that the strict transform \(Z'\) of \(Z\) in \(X'\) is smooth, and is \(D'=Z'\cap \pi^{-1}(D)\) is a simple normal crossings divisor in \(Z'\). Let \(M'\) be the generic pullback of \(M\) to \(Z'\), considered as a Hodge module on \(X'\). By the previous section, \(a_{X'+}M'\) satisfies the conclusions of the decomposition theorem, hence \autocite[Section 6.3]{Mochizuki2022} applies to the associated Twistor \(\mathscr D\)-modules, so hard Lefschetz holds for \(\Tp {a_{X}}M\). Note that since we only cite Mochizuki's work for the hard Lefschetz theorem, which is a statement on the level of perverse sheaves, we do not need to check any compatibility between Hodge modules and Twistor \(\mathscr D\)-modules.

One could alternatively use the results of \autocite{Shentu2021} to get the hard Lefschetz theorem.
\subsection{Rational decomposition}
\label{sec:org195833c}
Let us finally show how \autoref{cor:rational-decomp} follows from \autoref{thm:decomp}. If \(F\) is a real field, let \(E=\mathbb{R}\), else let \(E=\mathbb{C}\); note that \(F\) is dense in \(E\) in the Euclidean topology.

Let \(K_{E}=K_{F}\otimes_F E\), and note that \(\mathbf{R}f_{\ast}K_{E}\) admits a decomposition in the derived category by the hard Lefschetz theorem in \autoref{thm:decomp}(3) and Deligne's theorem \autocite{Deligne1968}. The following lemma implies the desired decomposition of \(\mathbf{R}f_{\ast}K_{F}\).
\begin{lemma}
Suppose \(C\) is a rational constructible complex such that \(C_{E}\) is isomorphic to \(\bigoplus_k{}^p\sh H^kC_{E}[-k]\) in the derived category. Then \(C\) is isomorphic to \(\bigoplus_k{}^p\sh H^kC[-k]\).
\end{lemma}
\begin{proof}
An isomorphism as required is a collection of maps \(f_k\colon{}^p\sh H^kC[-k]\to C\) in the derived category such that \({}^p\sh H^kf_k\) is an automorphism of \({}^p\sh H^kC\) for each \(k\).

Consider, for each \(k\), the linear map
\begin{align}
\hom({}^p\sh H^kC[-k],C)\xrightarrow{{}^p\sh H^k}\hom({}^p\sh H^kC,{}^p\sh H^kC).
\end{align}
Now the isomorphisms in \(\hom({}^p\sh H^kC,{}^p\sh H^kC)\) form a Zariski open subset: To check whether an endomorphism \(\phi\) is an isomorphism, it suffices to take a point \(x_i\) in each stratum of a stratification adapted to \({}^p\sh H^kC\), and test whether \(\phi_{x_i}\) is an isomorphism for each \(i\), giving finitely many open conditions on \(\phi\).

From the hypotheses of the lemma it follows that
\begin{align}
\hom({}^p\sh H^kC_{E}[-k],C_{E})\xrightarrow{{}^p\sh H^k}\hom({}^p\sh H^kC_{E},{}^p\sh H^kC_{E})
\end{align}
has an isomorphism in its image. By deforming slightly the real coefficients of a map \(f_{E,k}\colon {}^p\sh H^kC_{E}[-k]\to C_{E}\) that induces an automorphism of \({}^p\sh H^kC_{E}\), we get a rational \(f_k\) as desired since \(\hom({}^p\sh H^kC_{E}[-k],C_{E})=\hom({}^p\sh H^kC[-k],C)\otimes_{F}E\).
\end{proof}
\section{Kollár's conjecture}
\label{sec:org9992a8d}
Kollár \autocite{Kollar1986a} conjectured that to any VHS \(V\) defined on a smooth Zariski open subset of a projective variety \(X\), there should exist a coherent sheaf \(S_X(V)\) on \(X\) satisfying a decomposition theorem and certain Kodaira-type vanishing theorems. In particular, we should have \(S_X(\mathbb{Z}_X)=\omega_X\), and if \(V\) is defined on the complement of a normal crossings divisor, then \(S_X(V)\) should be the lowest piece of the extension of the Hodge filtration on \(V\) to Deligne's canonical extension of \(V\) to a vector bundle with logarithmic connection on \(X\).

The construction and decomposition theorem for the sheaves \(S_X(V)\) follow from \autoref{thm:decomp} as in \autocite[Theorem 3.2]{Saito1991}, as follows. Given \(V\) a generically defined polarizable VHS of weight \(n\) which admits an integral structure, let \(M\in \operatorname{HM}_{\mathbb{Z}}(X,n+\dim X)^p\) be the corresponding Hodge module. Then we define \(S_X(V)=F^{p(M)}M\) where \(p(M)\) is the largest integer \(p\) such that \(F^pM\ne 0\); note that \(F^pM\) is a coherent sheaf on \(X\).

\begin{theorem}
\label{thm:kollar}
Given a surjective proper Kähler morphism \(f\colon X\to Y\) and a generically defined polarizable VHS \(V\) of weight \(n\) on \(X\) which admits an integral structure, we get a decomposition in \(D^b \operatorname{Coh}(Y)\)
\begin{align}
\mathbf{R}f_{\ast}S_X(V)=\bigoplus_{p(V^i)=p(V)}S_Y(V^i)[-i]
\end{align}
where \(H^i\) is the direct image \(\mathbf{R}f_{\ast}^iV\) restricted to a Zariski open subset of \(Y\) over which \(f\) is a submersion, hence \(H^i\) is a VHS of weight \(i+n\).

In particular, \(\mathbf{R}f_{\ast}^iS_X(V)=S_Y(V^i)\) if \(p(V^i)=p(V)\) and \(0\) otherwise; and \(\mathbf{R}f_{\ast}^iS_X(V)=0\) if \(i\) is strictly greater than the dimension of a general fibre of \(f\).
\end{theorem}
\begin{proof}
Given \autoref{thm:decomp}, Saito's argument in \autocite{Saito1991} goes through in the Kähler case, see \autocite[Remark 2.7]{Saito1991}.

The key point, besides the decomposition theorem, is that if \(M\) is the Hodge module associated to \(V\) with strict support on \(X\), and if \(N\) is a summand of \(\Tp f^iM\) with strict support in a proper subvariety of \(Y\), then \(p(N)<p(M)\) \autocite[Proposition 2.6]{Saito1991}, hence only the summands of \(\Tpc f iM\) with strict support \(Y\) can contribute.
\end{proof}

The constant case \(V=\mathbb{Z}_X\) follows also from \autocite{Saito2022}, and the general case (for arbitrary polarizable VHS) from \autocite{Mochizuki2022}. There is also a more directly analytic approach to the theorem, and an extension to the twistor case, by Shentu and Zhao \autocite{Shentu2021b,Shentu2021a}.

\printbibliography
\end{document}